\documentclass[a4paper,12pt,centertags]{amsart}
\usepackage{mathrsfs}
\usepackage{amssymb}
\usepackage{bbm}
\usepackage[a4paper,centering]{geometry}
\usepackage[pdfdisplaydoctitle,colorlinks,breaklinks,urlcolor=blue,linkcolor=blue,citecolor=blue]{hyperref}
\usepackage{mathtools}
\usepackage{mathscinet}
\usepackage{tikz}
\usetikzlibrary{patterns}
\newtheorem{theorem}{Theorem}[section]
\newtheorem{lemma}[theorem]{Lemma}
\newtheorem{proposition}[theorem]{Proposition}

\theoremstyle{definition}

\newtheorem{assumption}[theorem]{Assumption}
\theoremstyle{remark}
\newtheorem{remark}[theorem]{Remark}
\numberwithin{equation}{section}
\tikzset{
  every path/.style={thick},
  not/.style={circle,fill=black,draw=black,inner sep=0pt,minimum size=0.5mm},
  >=stealth,
}
\protected\def\tzero{
  \tikz[baseline=0.5,scale=0.15]{
    \draw (0,1) node[not] {} -- (0,0);
  }
}
\protected\def\tuno{
  \tikz[baseline=0.5,scale=0.15]{
    \draw (-0.5,1) node[not] {} -- (0,0);
    \draw (0.5,1) node[not] {} -- (0,0);
  }
}

\protected\def\tunodue{
  \tikz[baseline=0.5,scale=0.15]{
    \draw (0,0.6) node {} -- (0,0);
    \draw (-0.5,1.2) node[not] {} -- (0,0.6);
    \draw (0.5,1.2) node[not] {} -- (0,0.6);
  }
}

\newcommand{\etazero}{\eta^{\tzero}}
\newcommand{\etauno}{\eta^{\mkern-2.5mu\tuno}}

\newcommand{\etaunodue}{\eta^{\mkern-4mu\tunodue}}

\newcommand{\e}{\operatorname{e}}
\DeclareMathOperator{\supp}{Supp}
\newcommand{\im}{\mathrm{i}}
\newcommand{\R}{\mathbf{R}}
\newcommand{\Z}{\mathbf{Z}}
\newcommand{\E}{\mathbb{E}}
\newcommand{\Prob}{\mathbb{P}}
\newcommand{\Tb}{\mathbb{T}}
\renewcommand{\Mc}{\mathcal{M}}
\newcommand{\Vc}{\mathcal{V}}

\newcommand{\Cs}{\mathscr{C}}

\newcommand{\Fs}{\mathscr{F}}

\newcommand{\Vs}{\mathscr{V}}
\newcommand{\Xs}{\mathscr{X}}
\newcommand{\Ys}{\mathscr{Y}}
\newcommand{\uno}{\mathbbm{1}}
\newcommand{\eqdef}{\vcentcolon=}

\newcommand{\scalar}[1]{\langle #1 \rangle}
\newcommand{\term}[1]{\text{\textcircled{\tiny #1}}}
\DeclareMathOperator{\plt}{\term{$<$}}
\DeclareMathOperator{\pgt}{\term{$>$}}
\DeclareMathOperator{\pgeq}{\term{$\geq$}}
\DeclareMathOperator{\peq}{\term{$=$}}
\def\MRnum#1\empty{#1}
\renewcommand{\MRhref}[2]{%
  \href{http://www.ams.org/mathscinet-getitem?mr=#1}{#2}
}
\renewcommand{\MR}[1]{
  \relax\ifhmode\unskip\space\fi
  \MRhref{\MRnum#1\empty}{\texttt{\Tiny[MR\MRnum#1\empty]}}
}
\hypersetup{%
  pdftitle={Random initial conditions for semi-linear PDEs},
  pdfauthor={D. Bloemker, G. Cannizzaro, M. Romito},
  pdfcreator={M. Romito}
}
\begin{document}
  \title{Random initial conditions for semi--linear PDEs}
  \author[D. Bl\"omker]{Dirk Bl\"omker}
    \address{Institut f\"ur Mathematik, Universit\"at Augsburg, D-86135 Augsburg, Germany}
    \email{\href{mailto:dirk.bloemker@math.uni-augsburg.de}{dirk.bloemker@math.uni-augsburg.de}}
  \author[G. Cannizzaro]{Giuseppe Cannizzaro}
    \address{Mathematics Institute, University of Warwick, Gibbet Hill Rd, CV4 7AL Coventry, United Kingdom}
    \email{\href{mailto:G.Cannizzaro@warwick.ac.uk}{G.Cannizzaro@warwick.ac.uk}}
  \author[M. Romito]{Marco Romito}
    \address{Dipartimento di Matematica, Universit\`a di Pisa, Largo Bruno Pontecorvo 5, I--56127 Pisa, Italia}
    \email{\href{mailto:marco.romito@unipi.it}{marco.romito@unipi.it}}
    \urladdr{\url{http://people.dm.unipi.it/romito}}
  \thanks{{M.\,R.} acknowledges the support of
    the Universit\`a di Pisa under the \emph{PRA – Progetti
    di Ricerca di Ateneo} (Institutional Research Grants)
    - Project no. PRA\_2016\_41 \emph{Fenomeni singolari
    in problemi deterministici e stocastici ed
    applicazioni}.}
  \keywords{random initial condition, semilinar PDEs}
  \date{July 10, 2017}
  \begin{abstract}
    We analyze the effect of random initial conditions
    on the local well--posedness of semi--linear PDEs,
    to investigate to what extent recent ideas on
    singular stochastic PDEs can prove useful in
    this framework.
  \end{abstract}
\maketitle
\section{Introduction}

This paper is a ``proof of concept'' that tries to
investigate the effect of random initial conditions
for the existence of partial differential equations
of evolution type.
These ideas have been pioneered by Bourgain \cite{Bou1994,Bou1996},
and recently there have been a lot of activity,
since the seminal papers by Burq and Tzvetkov \cite{BurTzv2008a,BurTzv2008b}.
We refer to the recent lecture notes of Tzvetkov
\cite{Tzv2016} for a more detailed account of the literature.

Most, if not all, of the literature focuses on
the interesting case of dispersive or hyperbolic
equations (with exceptions, see for instance
\cite{NahPavSta2013}). On the other hand in
that case the intrinsic difficulties of the
problems examined may hide the limitations
and features of the method we are analyzing.

Here we focus on semi--linear PDEs, because the theory on the
linear propagator is well established and do not obfuscate
the issues derived by the random initial condition method. Our
aim is thus to shed light on the possibilities and limitations
of the method.

The main subject of our investigation is a semi--linear
PDE with a simple linear operator (think of Laplacian
or bi--Laplacian operator), and a polynomial nonlinearity,
and we expect that the equation satisfies some kind
of scaling invariance. The idea is that this class
of equations represents, at first order, a general
class of fundamental equations.
In other words, we are interested in fundamental
characteristics, so we focus on homogeneous
nonlinearities, that ensure scaling laws.  

Scaling invariance gives an indication on which spaces
we can expect to solve the equation by a fixed point
argument. It is a well understood fact (see for
instance \cite{dispwiki}) that a
critical space of initial conditions is a space
whose norm is left invariant by the scaling of
the equation. Continuity of the nonlinearity
in sub--critical spaces (i.\,e., smaller than
a critical space) is not prevented by scaling,
and thus in such spaces a fixed point strategy
is expected to be successful (when only using
multi--linear estimates)

We analyse the problem in the class of (negative)
H\"older spaces. On the one hand they provide
the largest critical spaces, on the other hand
in such spaces there is no apparent gain in
using a Gaussian randomization of the initial
condition. Indeed, for a Gaussian random variable,
summability for every $p\geq1$ comes for free once
one knows that summability holds for at least one
exponent.

A full account of the general strategy considered
is given in the next section. In short,
we decompose the solution in the linear propagator
on the initial condition (that, for rough initial
condition, should capture all the degrees of
irregularity of the solution) and a (hopefully
smoother) remainder. This provides a new equation
for the remainder, where a new term is the
nonlinearity computed in the ``rough'' term.
Thus the main feature of the
random initial condition is to tame
the ``roughness'' of this term and make it
well defined for a wider range of the parameters.
Here regularity/roughness should be understood
in terms of singularity at $t=0$ (all these
functions are smooth when $t>0$).

In the setting we have described, we are thus able to answer
to a series of questions that we believe are relevant for
the subject.
\begin{enumerate}
  \item\textbf{Is a random initial condition useful (in this setting)?}
    The general strength of the method has been already
    established
    in the literature we have cited before. In this setting
    the method is effective in a series of examples (see the next
    section), namely we prove {a.\,s.} existence of local solutions
    with respect to suitable Gaussian measures supported over function
    spaces larger than those available through a standard fixed
    point argument.
  \item\textbf{When is it useful?}
    The validity of the method
    is graded though by a ratio between the linear and
    non--linear part of the equation (our parameter $\delta$ from
    Assumption~\ref{a:nonlinear_assumption}, that is,
    roughly speaking, the ratio between the largest order of derivative
    of the non--linear term, and the largest order of derivative of
    the linear term). The larger is the ratio, the lower is the
    validity.
  \item\textbf{Are initial distributions supported on spaces of
    super--critical initial conditions possible?}
    Unfortunately the method does not allow to prove
    results for super--critical data. Our analysis on
    semi--linear PDEs allows to set the analysis on
    H\"older spaces of negative order, that are
    essentially the largest critical spaces.
    We do not get results outside such spaces.
    
    A simple explanation is that, as already explained,
    critical spaces are determined by the scaling
    properties of the problem. By randomizing the
    initial condition we do not introduce any
    additional argument that ``breaks'' the scaling
    invariance.
    
    We point out that the situation is in a way
    different when dealing with dispersive/hyperbolic
    equations, where the properties we know of the
    linear propagator do not allow to set the analysis
    in arbitrary function spaces. 
  \item\textbf{May a second order (or beyond) expansion
    be useful?}
    As illustrated in the next section, the randomization
    is exploited by decomposing the solution in a
    ``irregular' term (the linear propagator computed
    over the random initial condition) and a ``smoother''
    remainder. The first term should capture the highest
    degree of irregularity of the solution. It is thus
    reasonable to believe that whenever the initial
    condition is ``very'' irregular, adding further
    additional terms in a, so-to-say, Taylor expansion,
    might be helpful. It turns out that in the setting
    of semi--linear PDEs this is not necessary, since
    the linear term (think of the Laplace operator)
    already makes the first term of the expansion
    super--smooth (the irregularity is read in terms
    of a singularity in time at $t=0$). In the setting
    of dispersive/hyperbolic equations, where the
    regularization of the linear problem is way much
    milder, additional terms in the expansion may be
    effective \cite{Tzv2016}.
  \item\textbf{Is renormalization needed?}
    In the recent theory on singular stochastic PDEs
    \cite{Hai2013,Hai2014,GubImkPer2015} some stochastic
    objects can be defined only when taking suitable infinities
    into account. Here no ``infinities'' are possible, since
    the linear propagator makes the stochastic objects
    smooth. Whenever a stochastic object cannot be defined,
    the random initial condition cannot fix the problem
    (see the example in Section~\ref{s:counter}).
    We notice though that for dispersive problem this
    is in general not the case and renormalization may
    prove useful (see for instance \cite{OhTzv2017})
  \item\textbf{Can we borrow further ideas from the theory
    of singular stochastic PDEs?}
    One of the deepest ideas in \cite{Hai2013,Hai2014,GubImkPer2015},
    that goes beyond the global decomposition first considered
    for such problems in the stochastic setting in
    \cite{DapDeb2002,DapDeb2003,DapDebTub2007}, is the
    local description of the degree of irregularity of a solution.
    In Section~\ref{s:local} we present a result based on the
    local description to prove local existence for logarithmically
    sub--critical initial conditions for the one-dimensional
    Burgers equation. We notice that a local description is
    useful only when the initial condition has regularity
    close to the critical level. We believe that this contribution
    is the main novelty of the paper.    
\end{enumerate}

\section{The main examples}

The examples we consider are equation on the $d$--dimensional
torus, with periodic boundary conditions, of the form
\begin{equation}\label{e:maineq}
  \partial_t u
    = Au + B(u),
\end{equation}
where $A$ is a linear operator and $B$ is a
multi--linear operator.
\subsection{The general strategy}

We expect that some
scaling invariance holds, that is there are $\sigma$,
$\tau$ such that if $u$ is solution of \eqref{e:maineq},
then so is
\begin{equation}\label{e:mainscaling}
  (t,x)\mapsto \lambda^\sigma u(\lambda^\tau t, \lambda x).
\end{equation}
\subsubsection{Deterministic initial condition}

We first consider \eqref{e:maineq}, with a deterministic
initial condition. We expect that the maximal Besov
space\footnote{more precisely, the maximal critical space
  where we are able to solve the equation is $\Vs^{\alpha,\beta}$,
  defined in formula \eqref{e:vanishing} (see also
  Remark~\ref{r:icc1}), with $\beta=1-\delta$,
  and $\delta>\frac12$.}
where we are able to solve our equation \eqref{e:maineq}
by means of a fixed point argument is $\Cs^{-\sigma}$, since
the homogeneous version of this space is invariant under
the transformation $u\leadsto \lambda^\sigma u(\lambda\cdot)$.

Assume the non--linear term $B$ is bi--linear and symmetric
(we will discuss more general cases in Sections~\ref{s:asymmetric}
and~\ref{s:cubic}),
and set
\[
  \Vc(u_1,u_2)(t)
    = \int_0^t \e^{A(t-s)}B(u_1,u_2)\,ds.
\]
A standard fixed point theorem (Theorem~\ref{t:fix1})
solves the above equation as
\[
  u
    = \eta_1 + \Vc(u,u),
\]
where $\eta_1=\e^{tA}u(0)$.
To this end we assume that the nonlinearity is suitably
continuous, namely there is $\delta\in[0,1)$ such that
\[
  \|\Vc(u_1,u_2)(t)\|_{\Cs^\alpha}
    \lesssim \int_0^t (t-s)^{-\delta}\|u_1\|_{\Cs^\alpha}\|u_2\|_{\Cs^\alpha}\,ds.
\]
By scaling (see Remark~\ref{r:fix_remark}), $\delta=1-\frac{\alpha+\sigma}\tau$.
The fixed point is solved in spaces $\Xs^{\alpha,\beta}_T$,
defined by means of the norm
\[
  \|\cdot\|_{\alpha,\beta,T}
    = \sup_{t\in[0,T]}t^\beta\|\cdot\|_{\Cs^\alpha},
\]
that encode the singularity at $t=0$.
The value $\alpha$ must be large enough that $\Vc$
has some continuity property. On the other hand the
larger is $\alpha$, the larger is $\beta$ to compensate
the difference in regularity with the initial condition.
We will choose $\alpha$ minimal to minimize the singularity.
Theorem~\ref{t:fix1} ensures now the existence of
a unique local solution as long as $\|\eta_1\|_{\alpha,\beta,T}\to0$
as $T\to\infty$, and $\beta<\frac12$, $\beta+\delta\leq1$.
If $\delta>\frac12$ the result is optimal and includes the
critical space.
Further improvements are only possible through some additional
information, such as a--priori estimates, that break
the scaling.

If on the other hand $\delta\leq\frac12$, the fixed point
theorem is not optimal and we can solve the problem
with initial conditions in $\Cs^r$ only for
$r>\alpha-\frac12\tau$ (see Remark~\ref{r:icc1}).
A possible strategy could be to single out $\eta_1$, the
most singular part of $u$. Set $u = v+\eta_1$, then
\begin{equation}\label{e:vfp}
  v
    = \Vc(v,v) + 2\Vc(v,\eta_1) + \eta_2,
\end{equation}
where $\eta_2=\Vc(\eta_1,\eta_1)$.
Unfortunately this does not really help without
a more detailed understanding of the nonlinearity
(see Remark~\ref{r:icc2}).
\subsubsection{Random initial condition}

We turn to random initial conditions.
For simplicity and to make our point clear,
we assume $u_0$ is a random field on the
torus with independent Gaussian Fourier
components. Regularity of $u_0$, as well
as of $\etazero=t\mapsto\e^{tA}u_0$ (here
we adopt Hairer's notation to make clear
that we deal with random objects), is
standard and does not give any advantage.

The crucial point is that randomness
plays its major role in taming the term
$\etauno=B(\etazero,\etazero)$,
and
in turn $\etaunodue=\Vc(\etazero,\etazero)$,
is well defined for a wider range of the
parameters (see Remark~\ref{r:betterstoch}).
To do these computations, we take some
simplifying assumptions, in particular
we assume that, at small scales,
$B$ is essentially
$D^a\bigl((D^b\cdot)(D^b\cdot)\bigr)$.

Since the random initial condition
has smoothened the singularity of
$\etaunodue$, it is now worthwhile
to apply the fixed point strategy
to the formulation~\ref{e:vfp}.
To do this, we need to check that the
assumptions of Theorem~\ref{t:fix2}
are met by the terms $\etauno$
and $\etaunodue$. To this aim,
Proposition~\ref{p:etazero_reg}
and Theorem~\ref{t:eta12_reg}
yield all the needed information.
We end up with a series of
inequalities over the parameters
$\tau,a,b$ that 
restrict the possible regularity
of the random field.

The first example we consider (surface growth) is one of those
where $\delta>\frac34$, thus the deterministic theory is
sufficient (as already known from \cite{BloRom2012}).
The second (KPZ) is borderline, since $\delta=\frac12$.
For the third (Kuramoto--Sivashinsky), the deterministic
theory is not sufficient to get initial conditions up to
the critical space, and this is only possible with random
initial conditions Finally, in the fourth example
(reaction-diffusion), not even random  initial conditions
are sufficient to catch the critical case.

Notice that the random initial condition method
always fails
because $\etauno$ has a singularity in time that is too
strong. In particular, going further to a second order
expansion does not help.

In the last part of the paper we shall give some
remarks and present some additional examples.
Since in the paper we will analyse 
mass--conservative, symmetric quadratic
nonlinearities, roughly speaking of the
form $D^a((D^b\cdot)^2)$ that allow for optimal
results, in Section~\ref{s:asymmetric}
we will discuss what happens in asymmetric case,
while in Section~\ref{s:cubic} we will look at
the case of nonlinearities with higher powers,
and finally in Section~\ref{s:nomass} we will
relax the constraint of mass conservation.

Section~\ref{s:counter} is somewhat different.
We will see there through an example that the fact
that even a random initial condition cannot in
general cover all cases up to the critical level
(as we shall see in the examples of
Sections~\ref{s:ks} and~\ref{s:rd})
is not a limitation of our proofs.

Finally, in Section~\ref{s:local} we present a result
that shows that, when dealing with (almost)
critical random initial conditions, a global
decomposition in terms of stochastic objects
and a remainder term as in \eqref{e:vfp},
is not sufficient. Our strategy is to
understand the local degree of irregularity
of the solution and to exploit this
fact to gain a tiny (logarithmic)
improvement that allows to close
the fixed point argument. This may
be seen as a glimpse of the
extremely sophisticated ideas
introduced in \cite{Hai2013,Hai2014,GubImkPer2015}.
Notice though that in the aforementioned papers
they use, roughly
speaking, two fundamental ideas: the first
is to understand the most irregular part of the
solution -- as we have done. The second is
to exploit again the probabilistic structure
to define the terms in the most irregular part
of the solution. This is apparently not needed
here.
\subsection{Surface growth}

Consider the following example (see \cite{BloRom2015}
for a general overview),
\[
  \partial_t u
    = -\Delta^2 u - \Delta u - \Delta|\nabla u|^2,
\]
with periodic boundary conditions and zero mean.
The equation has scaling invariance according
to formula \eqref{e:mainscaling}, with exponents
$\sigma=0$ and $\tau=4$ (the lower order
term is neglected for the invariance). Thus the
\emph{critical} space for fixed point is at
the level of $\Cs^0$ (more precisely,
$\Vs^{-1,1/4}$, defined in formula~\eqref{e:vanishing}).

Assumption~\ref{a:nonlin} holds with values
$a=2$, $b=1$, and Assumption~\ref{a:nonlinear_assumption}
holds with $\alpha=1$, $\delta=\frac34$.
Notice that the choice of $\alpha$ is the
minimal value that gives sense to the
non--linear term. The standard fixed
point result, Theorem~\ref{t:fix1},
holds sharp for initial conditions
in $\Cs^\gamma$ with $\gamma\geq0$.
The argument yielding the critical space has been
given also in \cite{BloRom2012}. We do not need
random initial condition here.
\subsection{KPZ}\label{s:kpz}

Consider the following problem on the torus,
\begin{equation}\label{e:kpz}
  \partial_t u
    = \Delta u - \Mc|\nabla u|^2,
\end{equation}
subject to periodic boundary conditions and zero mean,
where $\Mc$ is the projector onto the zero mean space,
namely
\[
  (\Mc z)(x)
    = z(x) - \int_{\Tb^d} z(y)\,dy.
\]
With additive noise this is a fundamental model in mathematical
physics, recently solved by Hairer \cite{Hai2013}.
The equation has scaling invariance with exponents
$\sigma=0$ and $\tau=2$. Thus the critical
space is at the level of $\Cs^0$ (more precisely
$\Vs^{0,1/2}$).
It can be easily checked that Assumption~\ref{a:nonlin}
holds with $a=0$, $b=1$, and that
Assumption~\ref{a:nonlinear_assumption}
holds with $\alpha=1$, $\delta=\frac12$, and again
$\alpha$ has been chosen as minimal.
Theorem~\ref{t:fix1}, holds
for initial conditions in $\Cs^\gamma$, with $\gamma>0$.
Unfortunately the critical
space $\Vs^{0,\frac12}$ cannot be captured neither by
the deterministic results (Theorem~\ref{t:fix1} and
Theorem~\ref{t:fix2}), nor by the random initial condition.
\subsection{KS}\label{s:ks}

Consider the following mass--conservative Kuramoto--Sivashinsky
equation
\[
  \partial_t u
    = -\Delta^2 u - \Delta u - \Mc|\nabla u|^2,
\]
with periodic boundary conditions and zero mean.
The scaling exponents
(when the lower order term $\Delta u$ is neglected)
are $\sigma=2$ and $\tau=4$, and the critical space for
fixed point is $\Cs^{-2}$. Assumption~\ref{a:nonlin}
holds again with $a=0$, $b=1$, while
Assumption~\ref{a:nonlinear_assumption} holds
with $\alpha=1$, $\delta=\frac14$, where
again $\alpha$ is the minimal number of derivatives
to give sense to the nonlinearity. Here
Theorem~\ref{t:fix1} holds for initial conditions
in $\Cs^\gamma$, with $\gamma>-1$, which is
still smaller than the critical space we have identified.
\subsection{Reaction--diffusion}\label{s:rd}

Consider the following equation with periodic boundary
conditions and zero mean,
\[
  \partial_t u
    = \Delta u - u + \Mc u^2,
\]
The scaling exponents are $\sigma=2$, $\tau=2$ (neglecting
as usual the lower order term), with critical space
$\Cs^{-2}$. Assumption~\ref{a:nonlin} holds
with values $a=0$, $b=0$, Assumption~\ref{a:nonlinear_assumption}
holds with $\delta=0$ and the minimal value $\alpha=0$.
Thus Theorem~\ref{t:fix1} applies for initial conditions
in $\Cs^\gamma$ with $\gamma>-1$.
\subsection{A short summary on Besov spaces} 

We will work with Besov spaces, which have somewhat maximal regularity
in terms of integrability. This or H\"older spaces are natural
spaces for the regularity of Gaussian random variables.

Besov spaces are defined via Littlewood-Paley projectors.
Let $\chi,\varrho$ be non--negative smooth radial functions such
that
\begin{itemize}
  \item The support of $\chi$ is contained in a ball and the support
    of $\varrho$ is contained in an annulus;
  \item $\chi(\xi)+\sum_{j\ge0}\varrho(2^{-j}\xi)=1$ for all
    $\xi\in\R^d$;
  \item $\supp(\chi)\cap\supp(\varrho(2^{-j}\cdot))
    = \emptyset$, for $j\geq 1$ and
    $\supp(\varrho(2^{-i}\cdot))\cap\supp(\varrho(2^{-j}\cdot))
    = \emptyset$ when $|i-j| > 1$.
\end{itemize}
Set $\varrho_j(x):=\varrho(2^{-j} x)$
for all $j\geq 0$ and $\varrho_{-1}(x):=\chi(x)$.
The Littlewood-Paley blocks are given in terms of the
discrete Fourier transform,
\[
  \Delta_j u
    = (2\pi)^{-d}\sum_{k\in\Z^d} \varrho_j(k) \Fs_{\Tb^d}(u)(k)e_k(x)
    = \sum_{k\in\Z^d} \varrho_j(k) u_k e_k.
\]
Let $\alpha\in\R$, $p,q\in[1,\infty]$, we define the Besov space
$B^\alpha_{p,q}(\Tb^d)$ as the closure of the space of
smooth periodic functions with respect to the norm
\[
  \|u\|^q_{B_{p,q}^\alpha(\Tb^d)}
    \eqdef \|(2^{j\alpha}\|\Delta_ju\|_{L^p(\Tb^d)})_{j\geq-1}\|_{\ell^q}
\]
We will mainly deal with the special case $p=q=\infty$,
so we introduce the notation
$\Cs^\alpha\eqdef B_{\infty,\infty}^\alpha(\Tb^d)$
and denote by
\[
  \|u\|_\alpha
    = \|u\|_{B_{\infty,\infty}^\alpha}
\]
its norm.
\subsubsection{The Bony paraproduct}

The Bony paraproduct $\plt$ is defined for
distributions $f,g$ with Littlewood--Paley
blocks $(\Delta_j f)_{j\geq-1}$ and
$(\Delta_j g)_{j\geq-1}$ as
\[
  f\plt g
    = \sum_{m\leq n-1} (\Delta_m f)(\Delta_n g).
\]
The term $f\pgt g$ is then defined as
$f\pgt g=g\plt f$, and the \emph{resonant}
term is defined by
\[
  f\peq g
    = \sum_{|m-n|\leq 1}(\Delta_m f)(\Delta_n g),
\]
so that, when the product makes sense,
$f\cdot g=f\plt g + f\peq g + f\pgt g$.
\section{The fixed point argument}\label{s:fixedpoint}

We outline here an abstract fixed point argument that
yields local existence and uniqueness for initial conditions
in the scale of H\"older--Besov spaces. The argument is
given in two flavours: standard and with rough initial
condition. To this end we state some assumptions on the
linear and non--linear part of the equation \eqref{e:maineq}
that capture the essential features of our examples
and that needed here.
\begin{assumption}[Schauder estimates]\label{a:schauder_assumption}
  The unbounded operator $A$ generates an analytic
  semigroup. Moreover there is $\tau>0$ such that
  the following estimates hold,
  \begin{equation}\label{e:schauder_assumption}
    \|\e^{tA}u\|_{\alpha+\beta}
      \leq ct^{-\frac\beta\tau}\|u\|_\alpha,
  \end{equation}
  for every $\alpha\in\R$, every $u\in B^\alpha_{\infty,\infty}$,
  and every $\beta\geq0$.
  \end{assumption}
Define the integrated nonlinearity,
\[
  \Vc(u_1,u_2)(t)
    = \int_0^t \e^{(t-s)A}B(u_1(s),u_2(s))\,ds.
\]
\begin{assumption}\label{a:nonlinear_assumption}
  There are $\alpha\in\R$, $\delta\in[0,1)$, and $c>0$ such that
  \begin{equation}\label{e:nonlinear_assumption}
    \|\Vc(u_1,u_2)(t)\|_\alpha
      \leq c\int_0^t (t-s)^{-\delta}
        \|u_1(s)\|_\alpha\|u_2(s)\|_\alpha\,ds
  \end{equation}
\end{assumption}
\begin{remark}\label{r:fix_remark}
  A few remarks on the assumptions above,
  \begin{itemize}
    \item it is fairly easy to check that the exponent
      on the right--hand side of \eqref{e:schauder_assumption}
      follows by a scaling argument if \eqref{e:mainscaling}
      holds;
    \item likewise, if \eqref{e:mainscaling} holds (and
      $\delta>0$), then again by scaling invariance
      $\delta=1-\frac{\alpha+\sigma}\tau$ if the inequalities
      are optimal;
    \item there is no apparent gain if we assume
      different norms for $u_1$, $u_2$ on the right--hand
      side of \eqref{e:nonlinear_assumption}. On the contrary,
      usually this gives a $\delta$ that depends on the
      smallest norm exponent.
  \end{itemize}
\end{remark}
Consider the equation \eqref{e:maineq} in its mild formulation,
\[
  u(t)
    = \e^{At}u_0
      + \Vc(u,u)(t).
\]
\subsection{The standard fixed point argument}

Given $u_0$, set $\eta_1(t)=\e^{At}u_0$.  We wish to solve
by fixed point the problem
\begin{equation}\label{e:fix1}
  u
    = \eta_1+\Vc(u,u),
\end{equation}
in the normed space
\[
  \Xs^{\alpha,\beta}_T
    \eqdef\{u:\|u\|_{\alpha,\beta,T}\eqdef\sup_{t\in[0,T]} t^\beta\|u(t)\|_\alpha<\infty\}\,.
\]
We immediately have the following proposition. 
\begin{proposition}\label{p:nlb}
  For $\beta<\frac12$ and $\delta + \beta\leq 1$,
  \[
    \|\mathcal{V}(u_1,u_2)\|_{\alpha,\beta,T}
      \leq cT^{1-\beta-\delta}
        \|u_1\|_{\alpha,\beta,T}\|u_2\|_{\alpha,\beta,T} 
  \]
\end{proposition}
\begin{proof}
Assumption \eqref{e:nonlinear_assumption} ensures

\[
  t^\beta\|\Vc(u_1,u_2)(t)\|_\alpha
    \leq ct^\beta
      \|u_1\|_{\alpha,\beta,T} \|u_2\|_{\alpha,\beta,T}
      \int_0^t (t-s)^{-\delta} s^{-2\beta}\,ds.
\]
Hence, as long as $\beta<\frac12$, the statement follows.
\end{proof}
The proposition above allows to verify the following theorem.
\begin{theorem}\label{t:fix1}
  Consider $\beta<1/2$ and $\beta\leq 1-\delta$, and assume
  $\|\eta_1\|_{\alpha,\beta,T}\to 0$ for $T\to0$. Then there
  is $T>0$ such that the equation \eqref{e:fix1} has a unique
  fixed point in $\Xs^{\alpha,\beta}_T$.
\end{theorem}
Notice that the initial conditions $u_0$ to whom the above
theorem applies are those such that
$\|t\mapsto\e^{tA}u_0\|_{\alpha,\beta,T}\to0$ as $T\to0$.
Let us denote by $\Vs^{\alpha,\beta}$ such a space, namely
\begin{equation}\label{e:vanishing}
  \Vs^{\alpha,\beta}
    = \bigl\{u_0: \lim_{T\to0}
      \|t\mapsto\e^{tA}u_0\|_{\alpha,\beta,T} = 0\bigr\}.
\end{equation}
The most interesting case is when $\beta=1-\delta$,
since $\Vs^{\alpha,\beta}$ becomes critical.
Indeed, a simple
computation shows that the norm $\|\cdot\|_{\alpha,\beta,T}$
is invariant by the scaling \eqref{e:mainscaling}, in the
sense that
\[
  \sup_{[0,T]}t^\beta
      [\lambda^\sigma u(\lambda^\tau t,\lambda\cdot)]_\alpha
    = \lambda^{\sigma+\alpha-\tau\beta}\sup_{[0,\lambda^\tau T]}
      t^\beta[u(t,\cdot)]_\alpha=\sup_{[0,\lambda^\tau T]}
      t^\beta[u(t,\cdot)]_\alpha,
\]
since, according to Remark~\ref{r:fix_remark}, $\sigma +\alpha=\tau\beta$ and where $[\cdot]_\alpha$ is the semi--norm of $\Cs^\alpha$.
\begin{remark}[Initial conditions in $\Cs^r$]\label{r:icc1}
  Another way to understand the computation above is to realize
  that the larger is $\beta$, the larger is the set of initial
  conditions. To see this we look for the minimal values of
  $\gamma$ such that $u_0\in\Cs^r$ yields
  $\eta_1\in\Vs^{\alpha,\beta}$
  for some $\beta$ compatible with the assumptions of the
  above theorem. By \eqref{e:schauder_assumption},
  \[
    \|\eta_1\|_{\alpha,\beta,T}
      \leq T^{\beta-\frac{\alpha-r}\tau}\|u_0\|_r,
  \]
  thus $\eta_1\in\Vs^{\alpha,\beta}$ when $r>\alpha-\beta\tau$.
  Therefore,
  \begin{itemize}
    \item if $\delta>\frac12$, we can take the maximal value
      $\beta=1-\delta$ and $u_0\in\Cs^r$ $\leadsto$
      $\eta_1\in\Vs^{\alpha,\beta}$ for all
      $r>-\sigma$. With this choice of the parameters,
      the space $\Vs^{\alpha,\beta}$ is critical.
    \item if $\delta\leq\frac12$, then we are
      restricted to $\beta<\frac12$ and we have
      $u_0\in\Cs^r$ $\leadsto$ $\eta_1\in\Vs^{\alpha,\beta}$
      under the sub--optimal condition $r>\alpha-\frac\tau2$.
  \end{itemize}
\end{remark}
\subsection{Fixed point with rough initial condition}\label{s:roughfp}

Suppose that the initial condition $u_0$ is too rough to apply
the results of the previous section. Set as before
$\eta_1=t\mapsto\e^{tA}u_0$ and let
\[
  v
    = u - \eta_1.
\]
Instead of \eqref{e:fix1}, this time we solve the
transformed problem
\begin{equation}\label{e:fix2}
  v
    = \Vc(v,v) + 2\Vc(v,\eta_1) + \eta_2,
\end{equation}
where we have set $\eta_2=\Vc(\eta_1,\eta_1)$.
The key argument is that by taking a random distribution
over the initial condition, Gaussian for instance, the
quadratic term $\eta_2$ is well defined, although
it could not be in principle defined in general using only
the regularity properties of $\eta_1$ or, more precisely,
the singularity at $t=0$.
We expect that the
mixed product $\Vc(v,\eta_1)$ might be fine, since $\eta_1$
is smooth away from $0$, and $v$ is zero at $0$.

We notice also that for the same reasons we expect
$\eta_2$ to be continuous at $0$, therefore
there is no need to use weighted spaces such as
$\Xs^{\alpha,\beta}_T$ with a positive exponent $\beta$.
In other words negative exponents are also allowed.
The following proposition is a minor modification of
Proposition~\ref{p:nlb}.
\begin{proposition}\label{p:nlb2}
  If $\beta + \gamma <1$ and $\delta + \gamma\leq 1$,
  then
  \[
    \|\Vc(u_1,u_2)\|_{\alpha,\beta,T}
      \leq c T^{1-\delta-\gamma}\|u_1\|_{\alpha,\beta,T}
        \|u_2\|_{\alpha,\gamma,T}.
  \]
\end{proposition}
Based on this proposition, we can prove the following theorem.
\begin{theorem}\label{t:fix2}
  Consider $\beta<\frac12$, $\beta+\delta\leq 1$,
  $\gamma+\beta<1$, and $\delta + \gamma\leq 1$.
  Given $u_0$, assume that
  $\|\eta_1\|_{\alpha,\gamma,T}\to 0$ and
  $\|\eta_2\|_{\alpha,\beta,T}\to 0$, for $T\to 0$.
  Then there is $T>0$ such that the problem \eqref{e:fix2}
  has a unique fixed point in $\Xs^{\alpha,\beta}_T$.
\end{theorem}
\begin{proof}
  The proof is again by fixed point argument and  very similar
  to the proof of Theorem \ref{t:fix1}. This time though
  we use a different weight in time, which might have a
  non-positive exponent, since the initial condition is $0$.
  
  For the self-mapping we use
  \[
    \|\Vc(v,v) + 2 \Vc(v,\eta_1) +
        \eta_2\|_{\alpha,\beta,T} 
      \leq c(\|v\|^2_{\alpha,\beta,T}
        + \|v\|_{\alpha,\beta,T}\|\eta_1\|_{\alpha,\gamma,T})
        + \|\eta_2\|_{\alpha,\beta,T}.
  \]
  Likewise, for the contraction property,
  \[
    \begin{aligned}
      \lefteqn{\|\Vc(v_1,v_1) + 2\Vc(v_1,\eta_1) - \Vc(v_2,v_2)
          + 2 \Vc(v_2,\eta_1)\|_{\alpha,\beta,T}}\qquad\qquad&\\
    &= \|\Vc(v_1+v_2,v_1-v_2)
       + 2\Vc(v_1-v_2,\eta_1)\|_{\alpha,\beta,T}\\
    &\leq c
       \bigl(\|v_1+v_2\|_{\alpha,\beta,T}
         +\|\eta_1\|_{\alpha,\gamma,T}\bigr)
         \cdot\|v_1-v_2\|_{\alpha,\beta,T},
    \end{aligned}
  \]
  by Propositions~\ref{p:nlb} and \ref{p:nlb2}.
\end{proof}
\begin{remark}[Initial conditions in $\Cs^r$]\label{r:icc2}
  In comparison with Remark~\ref{r:icc1}, we look for
  initial conditions $u_0\in\Cs^r$ such that the assumptions
  of Theorem~\ref{t:fix2} hold. By the (scaling--wise
  optimal) estimate \eqref{a:nonlinear_assumption}
  we have
  \[
    \|\eta_2(t)\|_\alpha
      \lesssim \int_0^t (t-s)^{-\delta}\|\eta_1(s)\|_\alpha^2\,ds
      \lesssim t^{1-2\gamma-\delta}\|\eta_1\|_{\alpha,\gamma,T},
  \]
  therefore $\|\eta_2\|_{\alpha,\beta,T}\leq
  T^{\beta+1-2\gamma-\delta}\|\eta_1\|_{\alpha,\gamma,T}^2$
  if $\gamma<\frac12$ and $2\gamma+\delta\leq\beta+1$.
  The same computation of Remark~\ref{r:icc1} yields
  $\eta_1\in\Xs^{\alpha,\gamma}$ if $u_0\in\Cs^r$ and
  $r\geq\alpha-\gamma\tau$.
  To minimize the value of $r$ we wish to take $\gamma$ as
  large as possible, thus $\gamma\approx\frac12$. Unfortunately
  this yields the condition $r>\alpha-\frac\tau2$, which is
  the same of Remark~\ref{r:icc1}.
  
  In conclusion the additional expansion has not given,
  at least for general initial conditions, any additional
  benefit.
\end{remark}
\section{Stochastic objects}

Here we discuss the existence and regularity of the
terms appearing in the fixed
point arguments of the previous section, when the
initial condition $u_0$ is a random variable  with
peculiar structure. Here we focus on the case of
bi--linear mass--conservative nonlinearity $B$,
we will comment later on the no--moving--frame case
and the need of renormalization.
\subsection{Diagonal (simplifying) assumptions}

Here we greatly simplify our problem~\eqref{e:maineq},
by assuming that the linear operator acts diagonally
on the Fourier basis, and that the non-linear
operator is a \emph{bona fide} product. The reason
is that we wish to exploit the decorrelations of
the random initial condition.

Let $(e_k)_{k\in\Z^d}$  be the standard
Fourier basis of the torus $\Tb_d$ of normalized complex
exponentials. 
\begin{assumption}\label{a:linear_diagonal}
  For every $k\in\Z^d$,
  $Ae_k = \lambda_k e_k$, with
  $\lambda_k\sim -c|k|^\tau$,
  for some constant $c>0$.
\end{assumption}
In the sequel we will assume, with no harm,
that $c=1$.
\begin{assumption}\label{a:nonlin}
  For each $k,m,n\in\Z^d$, let $B_{kmn}$ denote
  the product $B_{kmn}=\scalar{B(e_m,e_n),e_k}$. Then $B$ is such that
  $B_{kmn}=0$ if $k\neq m+n$ and
  \[
  B(u,v)
  = \sum_{k\in\Z^d}\sum_{m+n=k}B_{kmn}u_mv_ne_k
  \]
  for $u = \sum_m u_m e_m$, $v=\sum v_n e_n$.
  Moreover
  \begin{itemize}
    \item (mass conservation) $B_{0mn}=0$ for
    all $m,n$,
    \item (regularity) there are numbers $a,b\geq0$
    such that for $m+n=k$,
    \begin{equation}\label{e:Bwavewise}
      |B_{kmn}|
        \leq c|k|^a|m|^b|n|^b.
    \end{equation}
  \end{itemize}
\end{assumption}
\begin{remark}\label{r:Bscaling}
  We do not assume that \eqref{e:Bwavewise} is sharp. As already
  pointed out in Remark~\ref{r:fix_remark} for the fixed point argument,
  the asymmetric case does not help. If on the other
  hand \eqref{e:Bwavewise} is
  sharp, a elementary scaling argument shows that $\sigma+a+2b=\tau$.
  Indeed, if $a,b$ are optimal, then roughly speaking
  $B\approx D^a\bigl((D^b\cdot)(D^b\cdot)\bigr)$ that
  scales as
  \[
    B(u_\lambda,u_\lambda)
      = \lambda^{2\sigma+a+2b} (B(u,u))_\lambda,
  \]
  where $u_\lambda(t,x)=\lambda^\sigma u(\lambda^\tau t,\lambda x)$.
  On the other hand $(\partial_t-A)u_\lambda=\lambda^{\tau+\sigma}
  (\partial_t u-Au)_\lambda$.

  Actually the same result could be directly obtained,
  starting from \eqref{e:Bwavewise}, by elementary
  paraproduct estimates as those in \cite{GubImkPer2015}.
  These estimates would provide also, together with
  \eqref{e:schauder_assumption}, a connection with
  Assumption~\ref{a:nonlinear_assumption}.
\end{remark}
\subsection{Random initial condition}

For simplicity, we give Gaussian structure to the
initial condition. Other distributions are possible
though, once one assume essentially sub--normality
and hyper--contraction.
\begin{assumption}\label{a:u0}
  The random variable $u_0$ is Gaussian with the
  following representation in Fourier modes,
  \[
    u_0(x)
      = \sum_{k\in\Z^d} \phi_k \xi_k e_k,
  \]
  where $(\xi_k)_{k\in\Z^d}$ is a family
  of centred complex valued Gaussian random variables
  such that $\bar\xi_k=\xi_{-k}$ for all
  $k$, and with covariance
  \[
    \E[\xi_{k_1}\bar{\xi}_{k_2}]
      = \uno_{\{k_1=k_2\}},
  \]
  and $(\phi_k)_{k\in\Z^d}$ is a sequence of ``weights''.
  Moreover,
  \begin{itemize}
    \item (mass conservation) $\phi_0(0)=0$,
    \item (regularity) there is $\theta\in\R$ such
    that $|\phi_k|\sim|k|^\theta$.
  \end{itemize}
\end{assumption}
\subsection{Regularity of the stochastic objects}\label{s:randreg}

Given a random initial condition $u_0$ as above, we set
\[
  \etazero
    = t\mapsto\e^{At}u_0,
      \qquad
  \etauno
    = B(\etazero,\etazero),
      \qquad
  \etaunodue
    = \Vc(\etazero,\etazero)
    = t\mapsto\int_0^t\e^{(t-s)A}\etauno(s)\,ds.
\]
In the rest of the section we study the regularity of
$\etazero$, $\etauno$, and $\etaunodue$ in H\"older--Besov
spaces.
\subsubsection{Regularity of $u_0$}

We start with the regularity of $u_0$. This follows
from standard results.
Indeed, by \cite[Theorem 6.3]{Kah1985},
it follows that there is $c>0$ such that
\[
  \Prob[\|\Delta_j u_0\|_{L^\infty}
      \geq c\sqrt{j}2^{\frac12(2\theta+d)j}]
    \lesssim 2^{-2j}.
\]
Then the first Borel--Cantelli lemma ensures that
there is a random number $C$ such that {a.\,s.},
\begin{equation}\label{e:u0reg}
  \|\Delta_j u_0\|_{L^\infty}
    \leq C\sqrt{j}2^{\frac12(2\theta+d)j}.
\end{equation}
In conclusion $\|u_0\|_\alpha$ is almost surely finite
(and with exponential moments) as long as
$\alpha<-\theta-\frac{d}2$.
Notice that this holds in general for \emph{sub--normal}
independent sequences $(\xi_k)_{k\in\Z^d}$ (since so is
for the results in \cite{Kah1985}). Here  a random
variable $X$ is sub--normal if
$\E[\e^{\lambda X}]\leq\e^{\lambda^2/2}$.

In the Gaussian case we can completely characterize the
regularity of $u_0$ as follows. An elementary computation
shows that $\E[\|u_0\|^2_{H^\alpha}]=\infty$ if
$\alpha\geq-\theta-\frac{d}2$. Hence by Fernique's theorem
$u_0\not\in H^\alpha$ {a.\,s.}. Finally,
$\Cs^{\alpha'}\subset H^\alpha$ if $\alpha'>\alpha$,
thus $u_0\not\in \Cs^\alpha$ {a.\,s.} for every
$\alpha>-\theta-\frac{d}2$.
\subsubsection{Regularity of $\etazero$}

We turn to study the regularity of $\etazero$ in terms
of spaces $\Xs^{\alpha,\beta}_T$. The previous
considerations and Assumption~\ref{e:schauder_assumption}
immediately yield the following result.
\begin{proposition}\label{p:etazero_reg}
  If $u_0$ is as in Assumption~\ref{a:u0}, then
  $\etazero\in\Xs^{\alpha,\gamma}_T$ for every $T>0$ 
  and all $\alpha$, $\gamma$ such that
  \begin{itemize}
    \item $\alpha< -(\theta+\frac{d}2)$, $\gamma=0$,
    \item $\alpha\geq -(\theta+\frac{d}2)$,
      $\gamma>\frac{\alpha+\theta+\frac{d}2}\tau$.
  \end{itemize}
\end{proposition}
Moreover, using Assumptions~\ref{a:linear_diagonal},
we immediately obtain
\[
  \|\etazero_t-\etazero_s\|_\alpha
    \lesssim s^{-\frac\epsilon\tau}(t-s)^{\frac\epsilon\tau}\|u_0\|_\alpha,
\]
if $\alpha<-(\theta+\frac{d}2)$, and similarly if
$\alpha\geq-(\theta+\frac{d}2)$. In conclusion we always
have $\etazero\in C([0,T];\Cs^\alpha)$ if
$\alpha<-(\theta+\frac{d}2)$, and $\etazero\in C((0,T];\Cs^\alpha)$
otherwise.
\begin{remark}
  We see here that a random initial condition
  does not give any advantage at the level
  of $\etazero$. Due to the assumptions of
  Theorems~\ref{t:fix1} and \ref{t:fix2},
  $\etazero$ will always be supported over
  critical spaces.
\end{remark}
\subsubsection{Regularity of $\etauno$}

The regularity of $\etauno$, or more precisely the
singularity in time at $t=0$, is a fundamental step.
Here Assumption~\ref{a:linear_diagonal} will play
a crucial role.

Since $(\xi_k)_{k\in\Z^d_0}$ is a sequence of
independent real standard Gaussian random variables,
we see immediately that $\etauno$ is in the
second Wiener chaos. Moreover, as we shall verify
below, the $0^\text{th}$--order component is
zero, therefore $\etauno$ is in the
homogeneous second Wiener chaos.
To prove that there is no $0^\text{th}$--order
component, we recall that the $0^\text{th}$--order
component is simply the expectation of
$\etauno$,
\[
  \begin{multlined}[.97\linewidth]
    \E[\etauno(t,x)]
      = \sum_{k\in\Z^d}\Bigl(\sum_{m+n=k}B_{kmn}\phi_m\phi_n
        \e^{-t(|m|^\tau+|n|^\tau)}\delta_{m+n=0}\Bigr)e_k(x) =\\
      = \Bigl(\sum_{m+n=0}B_{0mn}|\phi_m|^2 
        \e^{-2t|m|^\tau}\Bigr)e_0(x)
      = 0,
  \end{multlined}
\]
by Assumptions~\ref{a:nonlin} and \ref{a:u0}.

For $\beta\geq0$,
\[
  \begin{aligned}
    \E[|\Delta_j\etauno|^2]
      &\lesssim\sum_{|k|\sim 2^j}\sum_{m+n=k}
        \varrho_j(k)^2 |B_{kmn}|^2|\phi_m|^2|\phi_n|^2 \e^{-2t(|m|^\tau+|n|^\tau)}\\
      &\lesssim t^{-2\beta}\sum_{|k|\sim 2^j}|k|^{2a}\sum_{m+n=k}
        \frac{|m|^{2\theta+2b}|n|^{2\theta+2b}}
          {(|m|^\tau+|n|^\tau)^{2\beta}}\\
  \end{aligned}
\]
The sum extended over all $m$, $n$ such that $m+n=k$
can be decomposed, by symmetry, in two sums over the
two sets $A_k=\{(m,n):m+n=k,|m|\geq|n|\geq\frac12|k|\}$
and $B_k=\{(m,n):m+n=k,|n|\leq\frac12|k|\leq|m|\}$.
For the sum over $A_k$, notice that on $A_k$ we have
$\frac13|m|\leq|n|\leq|m|$, thus, whatever is the sign
of $2\theta+2b$,
\[
  \sum_{A_k}\frac{|m|^{2\theta+2b}|n|^{2\theta+2b}}
      {(|m|^\tau+|n|^\tau)^{2\beta}}
    \lesssim\sum_{|m|\geq\frac12|k|}\frac1
      {|m|^{2\beta\tau-4b-4\theta}}
    \lesssim\frac1{|k|^{2\beta\tau-4b-4\theta-d}}.
\]
Here we need $2\beta\tau-4b-4\theta>d$, otherwise the sum would
diverge.

For the sum over $B_k$, notice that we also have
$|m|\leq\frac32|k|$, thus whatever is the sign
of $2b+2\theta-2\beta\tau$,
\[
  \sum_{B_k}\frac{|m|^{2\theta+2b}|n|^{2\theta+2b}}
      {(|m|^\tau+|n|^\tau)^{2\beta}}
    \lesssim\sum_{B_k}|m|^{2b+2\theta-2\beta\tau}|n|^{2b+2\theta}
    \lesssim |k|^{2b+2\theta-2\beta\tau}
      \sum_{|n|\leq|k|}|n|^{2b+2\theta}
\]
It is a standard fact to see that the sum on the right hand
side of the formula above behaves as $|k|^{(2b+2\theta+d)\vee0}$
(and as $\log |k|$ if $2b+2\theta=-d$).

In conclusion we need $2\beta\tau-4b-4\theta>d$, and in
that case,
\[
  \begin{aligned}
    \E[|\Delta_j\etauno|^2]
      &\lesssim t^{-2\beta}\sum_{|k|\sim 2^j}
        |k|^{2a-2\beta\tau+2b+2\theta+(2b+2\theta+d)\vee0}\\
      &\lesssim t^{-2\beta}
        2^{j(2a-2\beta\tau+2b+2\theta+(2b+2\theta+d)\vee0)},
  \end{aligned}
\]
with a multiplicative correction term of order $j$ (that
does not change our conclusions below) in the case
$2b+2\theta=-d$.
Therefore, by \cite[Lemma A.9]{GubImkPer2015} it follows that
\[
  \sup_t t^{2\beta}\E[\|\etauno\|_\alpha^2]
    <\infty
\]
for $\alpha<\beta\tau-\chi_1$, with $\beta\tau>\chi_0$,
where
\begin{equation}\label{e:chi}
  \chi_0
    = 2b+2\theta+\tfrac12d,
      \qquad
  \chi_1
    = a+b+\theta+(b+\theta+\tfrac12d)_+.
\end{equation}
By hyper--contractivity in the second Wiener chaos
\cite{Nel1973,Sim1974}, the following result follows.
\begin{lemma}\label{l:bound1}
  If $\alpha,\beta\in\R$ are such that
  \[
    \beta
      > \beta_0(\alpha)
      \eqdef\Bigl(\frac{\alpha+\chi_1}\tau\Bigr)
        \vee\Bigl(\frac{\chi_0}\tau\Bigr)_+,
  \]
  then for every $p\geq1$,
  \[
    \sup_{t\geq0}\E[(t^{\beta}\|\etauno\|_\alpha)^p]
      <\infty
  \]
\end{lemma}
\begin{remark}\label{r:betterstoch}
  The advantage of the random initial condition
  emerges here, as we see that we have a milder
  singularity at $t=0$. For comparison,
  let $u_0$ be a non--random initial condition and
  set, as in Section~\ref{s:fixedpoint},
  $\eta_1=t\mapsto\e^{tA}u_0$. We wish
  to find initial conditions
  where the minimal singularity in time of the
  Littlewood--Paley block of $B(\eta_1,\eta_1)$
  is worse than the one of random initial conditions.
  To this aim,
  assume that $B_{kmn}\approx |k|^a|m|^b|n|^b$
  and that the Fourier coefficients of $u_0$ are
  so that $u_0(k)\approx |k|^\theta$. Then
  \[
    |\Delta_j B(\eta_1,\eta_1)
      \approx \sum_{|k|\sim 2^j}|k|^a\sum_{m+n=k}
        |m|^{b+\theta}|n|^{b+\theta}
        \e^{-t(|m|^\tau+n^\tau)},
  \]
  and, for each $k$,
  \[
    \begin{multlined}[.95\linewidth]
      \sum_{m+n=k} |m|^{b+\theta}|n|^{b+\theta}
          \e^{-t(|m|^\tau+n^\tau)}
        \gtrsim\sum_{A_k} |m|^{b+\theta}|n|^{b+\theta}
          \e^{-t(|m|^\tau+n^\tau)}\\
        \approx\int_{|k|}^{+\infty}\rho^{2b+2\theta+d-1}
          \e^{-2t\rho^\tau}\,d\rho
        \gtrsim\int_{|k|}^{t^{-1/\tau}}\rho^{2b+2\theta+d-1}\,d\rho
        \sim t^{-\frac{2b+2\theta+d}\tau},
    \end{multlined}
  \]
  where $A_k$ is as above.
\end{remark}
\subsubsection{Regularity of $\etaunodue$}

By means of Assumption~\ref{a:schauder_assumption},
we can prove that $\etaunodue$ is in a $\Xs^{\alpha,\beta}_T$
space (actually a $\Vs^{\alpha,\beta}$ space) for suitable
values of $\alpha$, $\beta$. We start by stating the following lemma.
\begin{lemma}\label{l:bound12}
  Assume $\frac{\chi_0}\tau<1$.
  For every $\alpha<2\tau-\chi_1$, every $\beta>\beta_0(\alpha)-1$,
  every $p\geq1$, and every $T>0$, there is a number $c_T>0$
  such that
  \[
    \sup_{t\in[0,T]}\E[(t^\beta\|\etaunodue\|_\alpha)^p]
      \leq c_T.
  \]
  Moreover, $c_T\to0$ as $T\downarrow0$.
\end{lemma}
\begin{theorem}\label{t:eta12_reg}
  Under Assumptions~\ref{a:linear_diagonal}, \ref{a:nonlin},
  and \ref{a:u0}, if $\frac{\chi_0}\tau<1$, then
  $\etaunodue\in\Xs^{\alpha,\beta}_T$
  for all $\alpha<2\tau-\chi_1$, all $\beta>\beta_0(\alpha)-1$,
  and all $T>0$. Moreover, for every $p\geq1$ and for the same
  values of $\alpha,\beta$,
  \[
    \E[\|\etaunodue\|_{\alpha,\beta,T}^p]
      <\infty.
  \]
  In particular, $\etaunodue\in\Vs^{\alpha,\beta}$, {a.~s.}, for the
  same values of $\alpha$, $\beta$.
\end{theorem}
\begin{proof}
  Notice preliminarily that it is sufficient to prove the
  statement when $\beta$ is close to $\beta_0(\alpha)-1$,
  since for $\epsilon>0$, $\|\cdot\|_{\alpha,\beta+\epsilon,T}
  \leq T^\epsilon\|\cdot\|_{\alpha,\beta,T}$.

  Our strategy to prove the theorem is to find $\gamma\in(0,1)$
  and $p\geq1$ such that $\gamma p>1$ and $t\mapsto t^\beta\etaunodue_t$
  is in $W^{\gamma,p}([0,T];\Cs^\alpha)$ (with all moments). By
  Sobolev's embeddings, this concludes the proof of the theorem.

  It follows from Lemma~\ref{l:bound12} that
  \[
    \E\Bigl[\int_0^T \|t^\beta\etaunodue\|_\alpha^p\,dt\Bigr]
      \leq c_T
  \]
  for all $\alpha<2\tau-\chi_1$ and $\beta>\beta_0(\alpha)-1$.
  It remains to analyze the increments.

  \underline{\textsl{Case 1}}.
  Consider first the case $\tau-\chi_1\leq\alpha<2\tau-\chi_1$. Here we have
  $\beta_0(\alpha)-1\in[0,1)$, so in view of the initial remark, it is
  not restrictive to assume that $\beta\in(\beta_0(\alpha)-1,1)$.
  Let $s\leq t\leq T$, then
  \[
    t^\beta\etaunodue_t - s^\beta\etaunodue_s
      = (t^\beta-s^\beta)\etaunodue_t + s^\beta
        (\etaunodue_t - \etaunodue_s).
  \]
  Consider the first term. It is elementary to see that for $\lambda\in[0,1]$
  (the case $\lambda=0$ is obvious, the case $\lambda=1$ follows by Taylor
  expansion, the intermediate cases by interpolation),
  \[
    t^\beta-s^\beta
      \lesssim t^{(1-\lambda)\beta}s^{(\beta-1)\lambda}(t-s)^\lambda,
  \]
  thus by Lemma~\ref{l:bound12},
  \[
    \begin{aligned}
      \E\int_0^T\int_0^t\frac{\|(t^\beta-s^\beta)\etaunodue_t\|_\alpha^p}
          {|t-s|^{1+\gamma p}}\,ds\,dt
        &\lesssim\int_0^T\int_0^t\frac{t^{\beta p(1-\lambda)-\beta_1p}
          s^{\lambda  p(\beta-1)}}{|t-s|^{1+(\gamma-\lambda)p}}
          \E[(t^{\beta_1}\|\etaunodue_t\|_\alpha)^p]\,ds\,dt\\
        &\lesssim\int_0^T t^{\beta p(1-\lambda)-\beta_1p}
          \int_0^t\frac{ds}{s^{\lambda p(1-\beta)}
          |t-s|^{1+(\gamma-\lambda)p}}\,dt\\
        &\lesssim\int_0^T t^{-p(\gamma+\beta_1-\beta)}\,dt,
  \end{aligned} 
  \]
  where $\beta_1\in(\beta_0(\alpha)-1,\beta)$, and we need
  $\lambda p(1-\beta)<1$, $1+(\gamma-\lambda)p<1$,
  and $p(\gamma+\beta_1-\beta)<1$. In the limit
  $\lambda\downarrow\gamma$ and $\beta_1\downarrow\beta_0(\alpha)-1$,
  we obtain the two conditions
  \begin{equation}\label{e:cg1}
    \gamma p(1-\beta)<1,
    \qquad
    p\bigl(\gamma - (\beta+1-\beta_0(\alpha))\bigr)<1.
  \end{equation}
  Similar considerations applied to the second term
  yield additional conditions on $\gamma$, $p$.
  These can be summarized as
  follows: given $\beta\in(\beta_0(\alpha)-1,1)$, find
  $\gamma\in(0,1)$ and $p\geq1$ such that
  \begin{equation}\label{e:cond_case1}
    \begin{gathered}
      \gamma
        <\tfrac1{1-\beta}\tfrac1p,\qquad
      \gamma-\tfrac1p
        <\beta+1-\beta_0(\alpha),\qquad
      \gamma-\tfrac1p
        >0,\\
      \gamma
        <2-\beta_0(\alpha),\qquad
      \frac1p
        > \bigl(\tfrac{\chi_0}\tau\bigr)_+ - \beta.
    \end{gathered}
  \end{equation}
  Notice that by the choice of $\beta$, we have that
  $\beta+1-\beta_0(\alpha)<2-\beta_0(\alpha)$ and
  $(\chi_0/\tau)_+ - \beta<1-\beta<2-\beta_0(\alpha)$.
  Figure~\ref{f:case1} shows the non--empty area of all values of
  $\gamma$ and $1/p$ that meet all the requirements.
  \begin{figure}[!h]
    \begin{tikzpicture}[x=60mm,y=40mm,thick]
      \draw[->] (0,0) -- (0,1.1) node [left] {\Tiny $\gamma$};
      \draw[->] (0,0) -- (1.1,0) node [below] {\Tiny $\frac1p$};
      \draw[dotted] (0,1) node [left] {\Tiny $1$} -- (1,1);
      \draw[dotted] (1,0) node [below] {\Tiny $1$} -- (1,1);
      \draw[dashed] (0,0) -- (0.45,1);
      \draw[dotted] (0.45,0) node [left,rotate=90] {\Tiny $1-\beta$} -- (0.45,1);
      \fill[pattern=north east lines,pattern color=red!70!black] (0,0) -- (0.45,1) -- (0,1);
      \draw[dashed] (0,0.2) node [left] {\Tiny $\beta+1-\beta_0(\alpha)$} -- (0.8,1);
      \fill[pattern=north west lines, pattern color=green!70!black] (0,0.2) -- (0.8,1) -- (0,1);
      \draw[dashed] (0,0) -- (1,1);
      \fill[pattern=north west lines,pattern color=orange] (0,0) -- (1,1) -- (1,0);
      \draw[dashed] (0,0.75) node [left] {\Tiny $2-\beta_0(\alpha)$} -- (1,0.75);
      \draw[dotted] (0.75,0) node [left,rotate=90] {\Tiny $2-\beta_0(\alpha)$} -- (0.75,1);
      \fill[pattern=vertical lines,pattern color=purple] (0,0.75) -- (0,1) -- (1,1) -- (1,0.75);
      \draw[dashed] (0.25,0) node [left,rotate=90] {\Tiny $\bigl(\tfrac{\chi_0}\tau\bigr)_+ - \beta$} -- (0.25,1);
      \fill[pattern=horizontal lines,pattern color=blue] (0.25,0) -- (0.25,1) -- (0,1) -- (0,0);
    \end{tikzpicture}
    \caption{The white area contains all values $(\gamma,1/p)$ that satisfy \eqref{e:cond_case1}.}
    \label{f:case1}
  \end{figure}
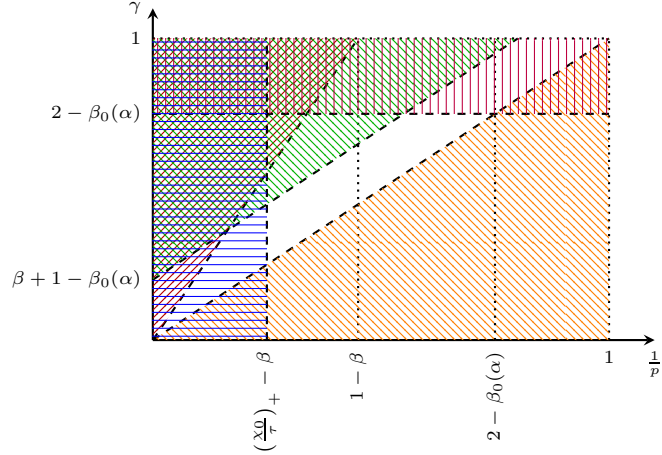
  
  \underline{\textsl{Case 2}}.
  Assume $\alpha<\tau-\chi_1$, then $\beta_0(\alpha)<1$ and, due to the initial remark,
  we can assume $\beta<0$. This time we decompose the increment as
  \[
    t^\beta\etaunodue_t - s^\beta\etaunodue_s
      = t^\beta(\etaunodue_t - \etaunodue_s)
        + (t^\beta-s^\beta)\etaunodue_s.
  \]
  Similar estimates as above yield the following conditions
  on $p$, $\gamma$:
  given $\beta\in(\beta_0(\alpha)-1,0)$, find
  $\gamma\in(0,1)$ and $p\geq1$ such that
  \begin{equation}\label{e:cond_case2}
    \beta_0(\alpha)p
      < 1,\qquad
    p\bigl(\gamma - (\beta+1-\beta_0(\alpha))\bigr)
      < 1.\qedhere
  \end{equation}
\end{proof}
\section{Additional examples}

\subsection{Non--symmetric nonlinearity}\label{s:asymmetric}

Our Assumption~\ref{a:nonlin} (as well as Assumption~\ref{a:nonlinear_assumption}
in the case of an optimal inequality) means essentially that $B$ is, at small scales,
like $D^a((D^b u)^2)$.
If this is not the case, the inequalities on which we base our
analysis are not optimal and the results are at most as good as those in the
symmetric case (that is, the critical level might not be achieved, even with
random initial conditions). A first order expansion though is still sufficient.

Consider for instance the following one--dimensional problem
\[
  du
    = Au + uu_{xx}.
\]
We can write $uu_{xx}=\frac12(u^2)_{xx}-(u_x)^2$, and notice that
the three terms $B(u,u)\eqdef uu_{xx}$, $B_1(u,u)\eqdef(u^2)_{xx}$,
and $B_2(u,u)=(u_x)^2$ scale with the same scaling, with $\sigma=\tau-2$.
Thus, using the theory detailed in these pages, we can solve the
problem in $\Xs^{1,\beta}_T$, for a suitable $\beta$. This is an
optimal choice for $B_2$, but not for $B_1$. This discrepancy
explains the non--optimal results in such cases.
\subsection{The case without mass conservation}\label{s:nomass}

We have worked so far under the assumption of mass conservation,
namely that the solution averages to zero in the spatial domain.
In this section we wish to briefly show that the general case
follows likewise, without too much hassle when 
mass conservation does not hold.

Consider $B$ quadratic, and let $U$ be solution of
\[
  \partial_t U
    = AU + B(U,U).
\]
Decompose $U=\xi+u$, where $\xi$ is the space average of $U$,
and $u$ has spatial mean zero. Recall that $\Mc$ is the projection
onto the zero mass space, so that $u = \Mc U$. 
Assume we work under Assumptions~\ref{a:linear_diagonal}
and~\ref{a:nonlin} (this time including the zero modes), then the
equations for $u$ and $\xi$ are
\[
  \begin{cases}
    \dot\xi
      = \Mc^\perp B(u,u) + B(\xi,\xi),\\
    \partial_t u
      = Au + \Mc B(u,u) + 2\Mc B(u,\xi),
  \end{cases}
\]
since $\Mc B(\xi,\xi)=0$ and $\Mc^\perp B(u,\xi)=0$.

We first notice that, if the initial condition has
infinite mean (this is for instance the case of a
distribution), there is in general no hope to
have a finite mean at positive times. We thus
consider in the rest of this section the case
of an initial condition with finite mean.

Assume, to fix ideas, that the numbers
$a,b$ are integers. We notice that if $a\geq1$, then $\Mc B=B$
and $\Mc^\perp B=0$, while if $b\geq1$ then $B(\xi,\cdot)=0$.
Moreover, $\Mc B(\Mc\cdot,\Mc\cdot)$ satisfies our original
Assumption~\ref{a:nonlin}(that is, a nonlinearity that
preserves the mass). We have three cases.
\begin{itemize}
  \item If $a\geq1$, then $\xi$ is a finite constant (in space
    and time) and the equation of $u$ is of the same kind we
    have studied so far, with the addition of the term of
    lower order $\Mc B(u,\xi)$ that does not change our
    analysis.
  \item If $a=0$, $b\geq1$, the equation for $u$ decouples from $\xi$,
    and is of the same kind we have studied so far. Once $u$ is
    known, then $\xi$ can be computed by the equation $\dot\xi=\Mc^\perp B(u,u)$. An additional difficulty is
    that if we solve the problem for $u$ in $\Xs_T^{\beta,b}$,
    then we cannot ensure that $\Mc^\perp B(u,u)$ is well
    defined. Indeed, for instance in the one--dimensional
    case (this is only to avoid ambiguity in the understanding
    of the generic term $D^b$),
    \[
      \Mc^\perp B(u,u)
        \sim \Bigl(\sum_{m+n=0} m^b n^b u_m u_n\Bigr)e_0
        \sim \|u\|_{H^b}^2,
    \]
    and $\Cs^b=B^b_{\infty,\infty}$ is not in $H^b=B^b_{2,2}$.
  \item Likewise, if $a=b=0$, the equation for $u$ contains
    the lower order term $\Mc B(u,\xi)$, while the equation
    for $\xi$ contains the polynomial term $B(\xi,\xi)$
    and again
    \[
      \Mc^\perp B(u,u)
        = \Bigl(\sum_{m+n=0}u_m u_n\Bigr)e_0
        \sim \|u\|_{L^2}
    \]
    with $L^2=B^0_{2,2}$.
\end{itemize}
In the last two cases a easy workaround is to solve the
problem in $\Xs^{b+\epsilon,\beta}_T$, since for $\alpha\geq0$,
$p\geq1$, and $\epsilon>0$,
\[
  \Cs^{\alpha+\epsilon}
    = B^{\alpha+\epsilon}_{\infty,\infty}
    \subset B^\alpha_{2,2}.
\]
\subsection{Higher powers in the nonlinearity}\label{s:cubic}

The overall picture provided by quadratic nonlinearities
does not change for non--linear terms with higher powers.
Indeed, assume $B$ is $m$--linear, with $m>2$, then
under an assumption analogous to
\eqref{e:nonlinear_assumption}, we see that if
$\delta>\frac1m$ then Theorem~\ref{t:fix1} is enough
for initial conditions up to (and including) the
critical space. If $\delta<\frac1m$ the random initial
condition method becomes effective and allows to solve the
initial value problem for rougher initial conditions
(but not as rough as the critical space in general).
We observe that also in the multi--linear case a first
order expansion is sufficient, because the method fails
for integrability of the analogous of $\etauno$ before
failing due to the smallness of (the analogous of) $\etaunodue$
in a suitable space.

Likewise, if we relax the condition of mass conservation
we can still solve the problem without having divergences
(so in the language of \cite{Hai2013}, there is no need to
include renormalization in the analysis).
\subsection{A counterexample}\label{s:counter}

Consider the following problem on $[-\pi,\pi]$ with
periodic boundary conditions, and zero mean,
\begin{equation}\label{e:ex1}
  \begin{cases}
    \partial_t u
      = u_{xx} + (u\star u)_x,
        \qquad x\in[-\pi,\pi],t\geq0,\\
    u\text{ is a odd function},
  \end{cases}
\end{equation}
where $\star$ denotes convolution on $(-\pi,\pi)$.
The equation has scaling invariance, with
$\tau=2$, $\sigma=2$, thus the critical space
is at the level of $\Cs^{-2}$.

In the rest of this section we show that we can find
(infinitely many) Gaussian
initial conditions $\Xi$ that are in $\Cs^{-\frac32-}$
{a.\,s.}, but such that there is no solution of the
above problem with initial condition $\Xi$.
\bigskip

The problem has a very simple formulation in Fourier
coordinates. A mean zero periodic odd function on
$[-\pi,\pi]$ has the Fourier expansion
\[
  u(x)
    = \sum_{k\in\Z}u_k\e^{\im kx}
    = 2\sum_{k=1}^\infty \xi_k\sin kx,
\]
with $u_k=-\im\xi_k$, $\xi_k\in\R$,
and $\xi_{-k}=-\xi_k$ for all $k$.
The equation, in terms of the new variables
$(\xi_k)_{k\geq1}$, is
\[
  \frac{d}{dt}\xi_k
    = -k^2\xi_k + k\xi_k^2,
      \qquad k\geq1.
\]
Each equation can be explicitly integrated, and by this one can easily see that
each component $\xi_k$ may blow up at the finite time
\begin{equation}\label{e:times}
  \tau_k
    = -\frac1{k^2}\log\Bigl(1 - \frac{k^2}{k\xi_k(0)}\Bigr)
\end{equation}
and we set $\tau_k=\infty$ if
the argument in the logarithm in
\eqref{e:times} is negative, or when the formula for $\tau_k$
gives a negative number. Elementary computations show that
$\tau_k<\infty$ when $\xi_k(0)>k$.

We have the following trichotomy
\begin{itemize}
  \item $\inf_{k\geq1}\tau_k = 0$: no local existence
    for \eqref{e:ex1},
  \item $\inf_{k\geq1}\tau_k>0$ and finite: local existence
    for \eqref{e:ex1},
  \item $\inf_{k\geq1}\tau_k=\infty$: global existence for
    \eqref{e:ex1}.
\end{itemize}
In view of the probabilistic argument, we notice that
$\inf_k\tau_k>0$ if and only if there is $\epsilon>0$
such that $\tau_k\geq\epsilon$ eventually.
\subsubsection{Random initial condition}

We consider as initial condition a Gaussian random
field $\Xi(x)=\sum_{k\geq1}\xi_k\sin kx$ with
independent $\xi_k$ with Gaussian law
$\mathcal{N}(0,\sigma_k^2)$.
\begin{lemma}
  If there are $\lambda>\sqrt2$ and $\epsilon>0$
  such that
  \[
    \sigma_k
      \leq \frac{k}{\lambda\sqrt{\log k}(1-\e^{-\epsilon k^2})},
        \qquad k\geq1,
  \]
  then $\inf_{k\geq1}\tau_k>0$, a.\,s. for the problem
  with initial condition $-\im\Xi$.
  Moreover $\Xi\in \Cs^{-\frac32-}$.
\end{lemma}
\begin{proof}
  The first part follows immediately by a Borel--Cantelli argument,
  since
  \[
    \sum_{k=1}^\infty \Prob[\tau_k\leq\epsilon]<\infty.
  \]
  Indeed,
  \[
    \begin{multlined}[.9\linewidth]
      \Prob[\tau_k\leq\epsilon]
        = \Prob\Bigl[\xi_k\geq\frac{k}{(1-\e^{-\epsilon k^2})}\Bigr]
        = \Prob\Bigl[Z\geq\frac{k}{\sigma_k(1-\e^{-\epsilon k^2})}\Bigr] \leq\\
        \leq \Prob[Z\geq\lambda\sqrt{\log k}]
        \lesssim \frac1{\lambda\sqrt{\log k}}\e^{-\frac12\lambda^2\log k}
        = \frac1{\lambda k^{\frac12\lambda^2}\sqrt{\log k}},
    \end{multlined}
  \]
  where $Z$ is a real standard Gaussian random variable.
  Therefore the series above converges since $\frac12\lambda^2>1$ by
  the choice of $\lambda$.

  To prove that $\Xi\in \Cs^{-\frac32-}$ we use Kolmogorov's continuity
  theorem. Indeed, let $E=(-\Delta)^{-1}\Xi$ (notice that the Laplace
  operator is invertible on the subspace of zero mean functions),
  then
  \[
    \begin{multlined}[.95\linewidth]
    \E[|E(x) - E(y)|^2]
      = \sum_{k=1}^\infty \frac{\sigma_k^2}{k^4}(\sin kx-\sin ky)^2\leq\\
      \leq \frac1{\lambda^2(1-\e^{-\epsilon})^2}\sum_{k=1}^\infty\frac{(1\wedge k|x-y|)^2}{k^2\log k}
      \lesssim |x-y|.
    \end{multlined}
  \]
  Since $E$ is Gaussian, we deduce that $E\in \Cs^{\frac12-}$ and therefore
  $\Xi\in \Cs^{-\frac32-}$.
\end{proof}
On the other hand, with the same regularity, we can show
an initial condition that gives non--existence.
\begin{lemma}
  Set
  \[
    \sigma_k
      = \frac{k}{\sqrt{2\log k}(1 - \e^{-k^2\epsilon_k})},
  \]
  with $\epsilon_k\downarrow0$. Then $\tau_k\leq\epsilon_k$ infinitely
  often, a.\,s. In particular $\inf_k \tau_k=0$ a.\,s. and there
  is no solution with initial condition $-\im\Xi$ with probability
  one. Moreover, if $\inf_k k^2\epsilon_k>0$, then
  $\Xi\in \Cs^{-\frac32-}$, a.\,s. 
\end{lemma}
\begin{proof}
  For the first part we use again a Borel--Cantelli argument.
  As above,
  \[
    \Prob[\tau_k\leq\epsilon_k]
      = \Prob[Z\geq\sqrt{2\log k}]
      \gtrsim\frac1{\sqrt{2\log k}}\e^{-\log k}
      = \frac1{k\sqrt{2\log k}},
  \]
  but this time the series diverges and $\tau_k\leq\epsilon_k$
  for infinitely many $k$ with probability one.
  
  The regularity follows as in the previous lemma, since
  for $E=(-\Delta)^{-1}\Xi$,
  \[
    \begin{multlined}[.9\linewidth]
    \E[|E(x) - E(y)|^2]
      = \sum_{k=1}^\infty \frac{\sigma_k^2}{k^4}(\sin kx-\sin ky)^2\leq\\
      \leq \frac1{2(1-\e^{-\delta})^2}\sum_{k=1}^\infty\frac{(1\wedge k|x-y|)^2}{k^2}
      \lesssim |x-y|.
    \end{multlined}
  \]
  where $\delta=\inf_k k^2\epsilon_k$.
\end{proof}
\section{A logarithmically sub--critical result}\label{s:local}

In this section we discuss the existence of solutions
with random initial conditions in the critical case.
We focus, as a standing example, on the Burgers equation
in dimension $d=1$, which is the equation for the
derivative of the solution of KPZ,
\begin{equation}\label{e:burgers}
  \partial_t u - u_{xx}
    = (u^2)_x.
\end{equation}
Notice that we have not changed the parameter $\delta$
from Assumption~\ref{a:nonlinear_assumption}. The
critical space on the other hand is (clearly)
shifted by one derivative.
\subsection{Setting of the problem}

\subsubsection{Random initial data}

We consider random initial data
as in Assumption~\ref{a:u0}, with
$\theta=1-\frac{d}2=\frac12$ and coefficients
\begin{equation}\label{e:logphi}
  |\phi_k|
    \sim |k|^\theta(\log(1+|k|))^{-\nu-\frac12}.
\end{equation}
Using \cite[Theorem 6.3]{Kah1985} as in formula
\eqref{e:u0reg}, we see that
\[
  \|\Delta_j u_0\|_\infty
    \leq C j^{-\nu}2^j,
\]
for a random constant $C$. Thus $\nu=0$ corresponds
to critical initial data, and $\nu>0$ to logarithmically
sub--critical initial data.
\subsubsection{The solution space}

We will solve the problem as in Section~\ref{s:roughfp}.
We set $u = v + \etazero$ and consider the problem
\begin{equation}\label{e:burgers2}
  \partial_t v - \partial_{xx}v
    = (v^2)_x + 2(v\etazero)_x + \etauno,
\end{equation}
The term $\etaunodue$, obtained by applying the heat
kernel to $\etauno$, has enough regularity for what we
will do. The troublemaker is $(v\etazero)_x$, since
given the regularity of $v$ and $\etauno$, the singularity
in time at $t=0$ is not integrable.
Before illustrating how to circumvent the problem,
we introduce the space where the problem will be solved.

Define the space $\Cs^\alpha_\kappa$ as the closure of smooth
functions with respect to the norm
\[
  \|u\|_{(\alpha,\kappa)}
    \eqdef\sup_{j\geq-1}(1+|j|^\kappa) 2^{\alpha j}
      \|\Delta_j u\|_\infty.
\]
This is as the space $\Cs^\alpha_\kappa$, but with a
logarithmically corrected growth. We state a few
properties of these spaces we shall need later.
To this end, define a tamed logarithm
$\ell:(0,\infty)\to\R$ as
\[
  \ell(t)
    = \log(\tfrac1t\vee 2).
\]
\begin{lemma}\label{l:cak_prop}
  The following properties hold,
  \begin{itemize}
    \item if $\alpha>0$ and $\kappa\in\R$, or if $\alpha=0$ and
      $\kappa>1$, then $\Cs^\alpha_\kappa\cdot \Cs^\alpha_\kappa
      \subset \Cs^\alpha_\kappa$,
    \item $\Cs^{\alpha+\epsilon}\subset \Cs^\alpha_\kappa
      \subset \Cs^\alpha$, for every $\epsilon>0$,
    \item if $\alpha'<\alpha$ and any $\kappa$, $\kappa'$, or
      if $\alpha=\alpha'$ and $\kappa\geq\kappa'$, then
      for every $t>0$ and $u\in \Cs^{\alpha'}_{\kappa'}$,
      \[
        \|\e^{t\Delta}u\|_{(\alpha,\kappa)}
          \lesssim t^{-\frac12(\alpha-\alpha')}
            \ell(t)^{\kappa-\kappa'}
            \|u\|_{(\alpha',\kappa')}.
      \]
  \end{itemize}
\end{lemma}
\begin{proof}
  For the first property, if $u,v\in \Cs^\alpha_\kappa$,
  with $\|u\|_{(\alpha,\kappa)}\leq1$, $\|v\|_{(\alpha,\kappa)}\leq 1$,
  \[
    \begin{multlined}[.95\linewidth]
      \|\Delta_j(u\plt v)\|_\infty
        \approx\Bigl\|\Delta_j\Bigl(\sum_{m=-1}^{j-2}(\Delta_m u)(\Delta_n v)\Bigr)\Bigr\|_\infty\lesssim\\
        \approx\sum_{m=-1}^{j-2}\|\Delta_m u\|_\infty\|\Delta_j v\|_\infty
        \lesssim j^{-\kappa}2^{-\alpha j}\sum_{m=-1}^{j-2} m^{-\kappa}2^{-\alpha m}
        \lesssim j^{-\kappa}2^{-\alpha j},
    \end{multlined}
  \]
  and
  \[
    \|\Delta_j(u\peq v)\|_\infty
      \lesssim\sum_{m=j}^\infty \|\Delta_m u\|_\infty\|\Delta_m v\|_\infty
      \lesssim \sum_{m=j}^\infty m^{-2\kappa}2^{-2\alpha m}
      \lesssim j^{-\kappa}2^{-\alpha j}.
  \]
  
  The second property is immediate by the definition of norms. For the third, using \cite[Proposition 2.4]{MouWebXu2016},
  \[
    \begin{multlined}[.8\linewidth]
      j^\kappa 2^{\alpha j}\|\Delta_j(\e^{t\Delta}u)\|_\infty
        \lesssim j^\kappa 2^{\alpha j}\e^{-2^{2j}t}
          \|\Delta_j u\|_\infty =\\
        = (j^{\kappa'}2^{\alpha'j}\|\Delta_ju\|_\infty)
          j^{\kappa-\kappa'}2^{(\alpha-\alpha')j}
          \e^{-2^{2j}t}
        \lesssim H_{\kappa-\kappa',\alpha-\alpha',2}(t)
          \|u\|_{(\alpha',\kappa')},
    \end{multlined}
  \]
  and the conclusion follows from Lemma~\ref{l:technical}.
\end{proof}
\subsection{A ``classical'' case}

Let us solve first a fixed point theorem for
\[
  \partial_t u
    = \Delta u + (u^2)_x,
\]
with a norm better suited for the critical level,
\[
  \|u\|_\bullet
    \eqdef\sup_{t\leq T}t^{\frac12}
      \ell(t)^a\|u(t)\|_{(0,\kappa)},
\]
with $\kappa>1$. Then by Lemma \ref{l:cak_prop},
\[
  \begin{multlined}[.9\linewidth]
    \|\Vs(u)\|_{(0,\kappa)}
      \leq\int_0^t \|\e^{(t-s)\Delta}(u^2)_x\|_{(0,\kappa)}\,ds
      \lesssim \int_0^t
        \|\e^{(t-s)\Delta}u^2\|_{(1,\kappa)}\,ds\\
      \lesssim\|u\|_\bullet^2\int_0^t
        (t-s)^{-\frac12}s^{-1}\ell(s)^{-2a}\,ds
      \lesssim t^{-\frac12}\ell(t)^{1-2a}\|u\|_\bullet^2,
  \end{multlined} 
\]
since by an elementary computation, if $\beta\in(0,1)$ and $a\geq0$, or $\beta=1$ and $a>1$, then
\[
  \int_0^t (t-s)^{-\frac12}s^{-\beta}\ell(s)^{-a}\,ds\lesssim
    \begin{cases}
      t^{\frac12-\beta}\ell(t)^{-a},
        \qquad&\beta<1,\\
      t^{-\frac12}\ell(t)^{1-a},
        \qquad&\beta=1.            
    \end{cases}
\]
Therefore, if $a>1$,
\[
  \|\Vs(u)\|_\bullet
    \leq \ell(T)^{-(a-1)}\|u\|_\bullet^2,
\]
Consider the initial condition. By Lemma \ref{l:cak_prop},
\[
  \|\e^{t\Delta}u(0)\|_{(0,\kappa)}
    \lesssim t^{-\frac12}\ell(t)^{\kappa-\kappa'}
      \|u(0)\|_{-1,\kappa'},
\]
thus
\[
  \|\e^{t\Delta}u(0)\|_\bullet
    \lesssim \ell(T)^{\kappa-\kappa'+a}\|u(0)\|_{(-1,\kappa')},
\]
if $\kappa-\kappa'+a<0$, that is $\kappa'>\kappa+a$.
This allows to prove  a fixed point theorem with initial
condition in $\Cs^{-1}_{\kappa'}$.

In view of a comparison with the results in the next
sections, consider an initial condition with
\[
  \|\Delta_j u(0)\|_\infty
    \sim j^{-\nu}2^j,
\]
then
\[
  \|u(0)\|_{(-1,\kappa')}
    =\sup_j j^{\kappa'}2^{-j}\|\Delta_j u(0)\|_\infty
    \sim\sup_j j^{\kappa'-\nu},
\]
is finite if $\nu\geq\kappa'$. Thus we have
$\nu\geq\kappa'>\kappa+a>2$. We will find
in Section~\ref{s:twist} below the condition
$\nu>1$, and in Section~\ref{s:slocal}
the condition $\nu>\frac12$).
\subsection{A ``classical'' case, with random
  initial condition}\label{s:twist}

Let $\Ys^{\kappa,\beta}_T$ be the space defined as
$\Xs^{0,\beta}$, but with the $\Cs^0$ norm replaced
by the $\Cs^0_\kappa$ norm.
By Theorem~\ref{t:eta12_reg} we know that
$\etaunodue\in\Xs^{0,\beta}_T$ (actually $\Vs^{0,\beta}$)
for $\beta>\beta_0(0)-1=\frac14$. The same argument
shows that $\etaunodue\in\Vs^{2\epsilon,\beta}$ for
$\beta>\frac14+\epsilon$, with $\epsilon>0$, thus
by Lemma~\ref{l:cak_prop} $\etaunodue\in\Ys^{\kappa,\beta}_T$
for $\beta>\frac14$ and $\kappa\geq0$, with
\begin{equation}\label{e:poly1}
  \|\etaunodue\|_{\Ys^{\kappa,\beta}_T}
    \leq C_{\tunodue} g_T,
\end{equation}
where $C_{\tunodue}$ is a random constant
and $g_T\lesssim T^\epsilon$ for a small
enough number $\epsilon>0$ (depending on the
value of $\beta$).

Moreover the previous lemma ensures that
for $\beta\in(\frac14,\frac12)$ and $\kappa>1$,
\begin{equation}\label{e:poly2}
  \|\Vc(v,v)\|_{\Ys^{\kappa,\beta}_T}
    \lesssim g_T\|v\|_{\Ys^{\kappa,\beta}_T}^2,
\end{equation}
with $g_T\lesssim T^\epsilon$ as above.
This shows that 
the term $(v\etazero)_x$ is the
``troublemaker'', as is the term that so far
prevents us to apply a fixed point theorem
to problem \eqref{e:burgers2} in
$\Ys^{\kappa,\beta}_T$

In this section we analyze the term $(v\etazero)_x$,
and show that if $\nu>1$, then $\Vs(v\plt\etazero)$
is well defined and the fixed point theorem strategy
can be completed.
\begin{proposition}\label{p:classical}
  Consider a random initial condition as in
  Assumption~\ref{a:u0}, with coefficients as
  in \eqref{e:logphi}. If $\nu>1$, then there
  is a random time $T$, with $T>0$ {a.\,s.},
  such that problem \eqref{e:burgers2} has
  a unique solution in $\Ys^{\kappa,\beta}_T$,
  where $\beta\in(\frac14,\frac12)$,
  and $\kappa\in(1,\nu]$.
\end{proposition}
\begin{proof}
  Let $v\in\Ys^{\kappa,\beta}_T$ with
  $\|v\|_{\Ys^{\kappa,\beta}_T}\leq 1$.
  We have
  \[
    \|\Delta_j(v\pgt\etazero)\|_\infty
      \lesssim \|\Delta_j v\|_\infty
        \sum_{n=0}^{j-2}\|\Delta_n\etazero\|_\infty
      \lesssim j^{-\kappa}t^{-\beta}
        \sum_{n=0}^{j-2}n^{-\nu}2^n\e^{-2^{2n}t},
  \]
  where we have used the fact that, to compute
  $\Delta_j(v\pgt\etazero)$, the relevant modes of $v$
  are those at levels $m\approx j$ (for simplicity of
  computations we have only considered $m=j$, but due
  to the estimates we have on $\Delta_j v$ and
  $\Delta_j\etazero$, the result is the same up
  to a multiplicative constant). Thus
  \begin{equation}\label{e:log1}
    \begin{aligned}
      t^\beta j^\kappa\Bigl\|\Delta_j
          \int_0^t\e^{(t-s)\Delta}(v\pgt\etazero)_x\,ds\Bigr\|_\infty
        &\lesssim t^\beta j^\kappa\int_0^t
          2^j\e^{-2^{2j}(t-s)}\|\Delta(v\pgt\etazero)\|_\infty\,ds\\
        &\lesssim t^\beta\sum_{n=2}^{j-2}n^{-\nu}2^n
          \int_0^t 2^j\e^{-2^{2j}(t-s)}\e^{-2^{2n}s}s^{-\beta}\,ds\\
        &\lesssim \sqrt{t}\sum_{n=2}^{j-2}n^{-\nu}2^n
          \e^{-2^{2n}t}\\
        &= \sqrt{t} G_{-\nu,1,2}(t).
    \end{aligned}
  \end{equation}
  From Lemma~\ref{l:technical} we know that
  $\sqrt{t} G_{-\nu,1,2}(t)$ is bounded if
  $\nu\geq0$, and converges to $0$ as $t\to0$
  if $\nu>0$. We notice that in
  particular we do not need the assumption on
  $\kappa$ here.

  Likewise,
  \[
    \begin{multlined}[.9\linewidth]
      \|\Delta_j(v\peq\etazero)\|_\infty
        \lesssim \sum_{n=j}^\infty\|\Delta_n v\|_\infty
          \|\Delta_n\etazero\|_\infty\lesssim\\
        \lesssim j^{-\kappa}t^{-\beta}
          \sum_{n=j}^\infty n^{-\nu}2^n\e^{-2^{2n}t}
        \lesssim j^{-\kappa}t^{-\beta}G_{-\nu,1,2}(t),
    \end{multlined}
  \]
  thus, by Lemma~\ref{l:technical},
  \begin{equation}\label{e:log2}
    \begin{aligned}
      t^\beta j^\kappa\Bigl\|\Delta_j
          \int_0^t\e^{(t-s)\Delta}(v\peq\etazero)_x\,ds\Bigr\|_\infty
        &\lesssim t^\beta j^\kappa\int_0^t
          2^j\e^{-2^{2j}(t-s)}\|\Delta(v\peq\etazero)\|_\infty\,ds\\
        &\lesssim t^\beta\int_0^t 2^j\e^{-2^{2j}(t-s)}
          s^{-\beta}G_{-\nu,1,2}(s)\,ds\\
        &\lesssim t^\beta\int_0^t (t-s)^{-\frac12}
          s^{-\beta}G_{-\nu,1,2}(s)\,ds
    \end{aligned}
  \end{equation}
  and, as before, it is sufficient to assume that $\nu>0$.
  
  Finally, since $\kappa>1$,
  \[
    \|\Delta_j(v\plt\etazero)\|_\infty
      \lesssim \|\Delta_j\etazero\|_\infty
        \sum_{n=0}^{j-2}\|\Delta_n v\|_\infty
      \lesssim j^{-\nu}2^j\e^{-2^{2j}t}
        \sum_{n=0}^{j-2} n^{-\kappa}t^{-\beta}
      \lesssim j^{-\nu}2^j\e^{-2^{2j}t}t^{-\beta},
  \]
  therefore
  \begin{equation}\label{e:badguy}
    \begin{aligned}
      t^\beta j^\kappa\Bigl\|\Delta_j
          \int_0^t\e^{(t-s)\Delta}(v\plt\etazero)_x\,ds\Bigr\|_\infty
        &\lesssim t^\beta j^\kappa\int_0^t
          2^j\e^{-2^{2j}(t-s)}\|\Delta_j(v\plt\etazero)\|_\infty\,ds\\
        &\lesssim j^{\kappa-\nu}2^{2j}t\e^{-2^{2j}t}\\
        &\lesssim t H_{\kappa-\nu,2,2}(t),
    \end{aligned}
  \end{equation}
  where $tH_{\kappa-\nu,2,2}(t)$ is bounded for $\kappa\leq\nu$,
  and $tH_{\kappa-\nu,2,2}(t)\to0$ for $\kappa<\nu$,
  by Lemma~\ref{l:technical}.
\end{proof}
\subsection{Local description}\label{s:slocal}

Consider the case $\nu\leq1$. From the proof of
Proposition~\ref{p:classical}, we see that there
is a random constant $C$, independent from $T\leq1$,
such that
\[
  \|\Vs(v\pgeq\etazero)\|_{\Ys^{\kappa,\beta}_T}
    \leq C g_{\nu,T}\|v\|_{\Ys^{\kappa,\beta}_T}
\]
for every $v\in \Ys^{\kappa,\beta}_T$, with
$g_{\nu,T}\downarrow0$ as $T\downarrow0$, and
in the above formula by $\Vs(v\pgeq\etazero)$
we mean that only the part $v\pgeq\etazero$ of
the product $v\etazero$ appears in $\Vs$.
Thus the
irregularity of a solution of \eqref{e:burgers2}
is due to the term $\Vs(v\plt\etazero)$.

For a given $v\in\Ys^{\kappa,\beta}$, set
\[
  R(v)
    \eqdef \Vs(v,v) + \etaunodue + 2\Vs(v\pgeq\etazero)
\]
\begin{lemma}\label{l:remainder}
  Let $u_0$ be a random field as in Assumption~\ref{a:u0},
  with coefficients as in \eqref{e:logphi}, and
  $\nu\in(\frac12,1]$. If $\kappa\in(1,2\nu)$
  and $v,v'\in\Ys^{\kappa,\beta}_T$, then
  \[
    t^\beta j^\kappa\ell(t)^\nu\|\Delta_j R(v)\|_\infty
      \lesssim (1 + \|v\|_{\Ys^{\kappa,\beta}_T})\|v\|_{\Ys^{\kappa,\beta}_T}^2.
  \]
  Moreover,
  \[
    \|\Vs(R(v)\plt\etazero)\|_{\Ys^{\kappa,\beta}_T}
      \lesssim \ell(T)^{\kappa-2\nu}
        (1 + \|v\|_{\Ys^{\kappa,\beta}_T}) \|v\|_{\Ys^{\kappa,\beta}_T}.
  \]
  and 
  \[
    \|\Vs(R(v)\plt\etazero) - \Vs(R(v')\plt\etazero)\|_{\Ys^{\kappa,\beta}_T}
      \lesssim \ell(T)^{\kappa-2\nu}
        (1 + \|v+v'\|_{\Ys^{\kappa,\beta}_T})\|v-v'\|_{\Ys^{\kappa,\beta}_T}.
  \]
  
\end{lemma}
\begin{proof}
  The first statement follows from \eqref{e:poly1}, \eqref{e:poly2},
  \eqref{e:log1}, and \eqref{e:log2}. For the second statement,
  for $v$ such that $\|v\|_{\Ys^{\kappa,\beta}_T}\leq1$,
  \[
    \begin{aligned}
      t^\beta j^\kappa\|\Delta_j\Vs(R(v)\plt\etazero)\|_\infty
        &\approx t^\beta j^\kappa\int_0^t 2^j\e^{-2^{2j}(t-s)}
          \|\Delta_j(R(v)\plt\etazero)\|_\infty\,ds\\
        &\approx t^\beta j^\kappa\int_0^t 2^j\e^{-2^{2j}(t-s)}
          j^{-\nu}2^j\e^{-2^{2j}s}\sum_{n=0}^{j-2}s^{-\beta}
          \ell(s)^{-\nu} n^{-\kappa}\\
        &\lesssim t H_{\kappa-\nu,2,2}(t)\ell(t)^{-\nu}\\
        &\lesssim \ell(T)^{\kappa-2\nu},
    \end{aligned}
  \]
  using Lemma~\ref{l:technical}, since we have chosen $\kappa<2\nu$.
  The third statement follows likewise.
\end{proof}
Our original equation, can be written as
\[
  v
    = 2\Vs(v\plt\etazero) + R(v),
\]
where we have understood that $R(v)$ is a
``smooth'' perturbation. The above equality
represents both our equation and a decomposition
of the solution in its regular and irregular part.
We thus replace $v$ with its decomposition in the irregular
part of the equation, to get
\[
  \begin{aligned}
    v
      &= 2\Vs(v\plt\etazero) + R(v)\\
      &= 2\Vs\bigl(((2\Vs(v\plt\etazero) + R(v))\plt\etazero)\bigr) + R(v)\\
      &= 4\Vs(\Vs(v\plt\etazero)\plt\etazero) + 2\Vs(R(v)\plt\etazero) + R(v)
  \end{aligned}    
\]
\begin{theorem}\label{t:local}
  Let $u_0$ be a random field as in
  Assumption~\ref{a:u0}, with coefficients as
  in \eqref{e:logphi}. If $\nu\in(\frac12,1]$, then there
  is a random time $T$, with $T>0$ {a.\,s.},
  such that the problem
  \begin{equation}\label{e:second}
    v
      = 4\Vs(\Vs(v\plt\etazero)\plt\etazero) + 2\Vs(R(v)\plt\etazero) + R(v)
  \end{equation}
  has a unique solution in $\Ys^{\kappa,\beta}_T$,
  where $\beta\in(\frac14,\frac12)$,
  and $\kappa\in(1,2\nu]$.
\end{theorem}
\begin{proof}
  Everything boils down to an estimate of
  $\Vs(\Vs(v\plt\etazero)\plt\etazero)$. All other terms are
  taken care of by Lemma~\ref{l:remainder}.
  Consider $v\in\Ys^{\kappa,\beta}_T$, with
  $\|v\|_{\Ys^{\kappa,\beta}_T}\leq 1$.
  The estimate
  \eqref{e:badguy} yields
  \[
    \|\Delta_j\Vs(v\plt\etazero)\|_\infty
      \lesssim j^{-\nu}2^{2j}t^{1-\beta}\e^{-2^{2j}t}.
  \]
  Thus, by Lemma~\ref{l:technical},
  \[
    \begin{aligned}
      \|\Delta_j(\Vs(v\plt\etazero)\plt\etazero)\|_\infty
        &\lesssim j^{-\nu}2^j\e^{-2^{2j}t}\sum_{n=0}^{j-2}
          \|\Delta_h\Vs(v\plt\etazero)\|_\infty\\
        &\lesssim j^{-\nu}2^j\e^{-2^{2j}t}\sum_{n=0}^{j-2}
          n^{-\nu}2^{2n}\e^{-2^{2n}t}t^{1-\beta}\\
        &=j^{-\nu}2^j\e^{-2^{2j}t}t^{1-\beta}H_{-\nu,2,2}(t)\\
        &\lesssim j^{-\nu}2^j\e^{-2^{2j}t}t^{-\beta}\ell(t)^{-\nu},
    \end{aligned}
  \]
  therefore,
  \[
    \begin{aligned}
      t^\beta j^\kappa\|\Delta_j\Vs(\Vs(v\plt\etazero)\plt\etazero)\|_\infty
        &\lesssim t^\beta j^{\kappa-\nu} 2^{2j}
          \int_0^t \e^{-2^{2j}(t-s)}s^{-\beta}\e^{-2^{2j}s}
          \ell(s)^{-\nu}\,ds\\
        &\lesssim t j^{\kappa-\nu} 2^{2j}\e^{-2^{2j}t}\ell(t)^{-\nu}\\
        &\leq t H_{\kappa-\nu,2,2}(t)\ell(t)^{-\nu}\\
        &\lesssim \ell(T)^{\kappa-2\nu}
    \end{aligned}
  \]
  This is sufficient to prove a fixed point theorem, with existence time
  dependent on the random constants in the above estimate and in
  Lemma~\ref{l:remainder}.
\end{proof}
\begin{remark}
  We wish to point out that the necessity of a local description
  to set out a problem amenable to a fixed point argument as we
  have done above, emerges only for initial conditions in
  (almost) critical spaces. Indeed, we know (see Remark~\ref{r:icc1})
  that if $\delta>\frac12$, then Theorem~\ref{t:fix1} is sufficient
  to find solutions with initial conditions in critical spaces.
  The challenge for random initial conditions rests in the case
  $\delta\leq\frac12$. If $\delta=\frac12$ the only open case
  is the critical case and can be sorted out as we have done in
  this section.
  
  Consider the case $\delta<\frac12$. It is not difficult to
  see that, as long as we require that the initial condition
  is sub--critical and $\etaunodue$ is
  well defined (that is $\etauno$ has a integrable singularity
  at $t=0$), then the term $\Vs(v,\etazero)$ makes sense in
  the right space and Theorem~\ref{t:fix2} provides a solution.
  The case of initial conditions in critical spaces is a
  different story. Here we need again the methods we have
  illustrated in this section. The computations are very
  similar.
\end{remark}
Finally, we wish to discuss to what extent the solution
provided by Theorem~\ref{t:local} is a solution of
problem \eqref{e:burgers2}, and in turns of problem
\eqref{e:burgers}.
\begin{proposition}
  Under the assumptions of Theorem~\ref{t:local} above, if
  $v$ is the solution defined on $[0,T]$ provided by the
  above--mentioned theorem,
  then \eqref{e:burgers2} holds in $\Xs^{0,\beta}_{T'}$
  for some {a.\,s.} positive random time $T'\leq T$.
\end{proposition}
\begin{proof}
  We give a quick sketch. Define the norm
  \[
    \|\cdot\|_{\kappa,\nu,\beta,T}
      \eqdef\sup_{[0,T]}t^\beta \ell(t)^\nu\|\Delta_j\cdot\|_\infty
  \]
  then the arguments in the proofs of Lemma~\ref{l:remainder} and
  Theorem~\ref{t:local} show that
  \begin{itemize}
    \item $\|w\|_{\kappa,\nu,\beta,T}<\infty\Longrightarrow
      \|\Vs(w\plt\etazero)\|_{\kappa,2\nu-\kappa,\beta,T}<\infty$,
    \item $w\in\Ys_T^{\kappa,\beta}\Longrightarrow
      \|\Vs(\Vs(w\plt\etazero)\plt\etazero)\|_{\kappa,2\nu-\kappa,\beta,T}<\infty$.
  \end{itemize}
  Thus, a solution of \eqref{e:second} satisfies
  $\|v\|_{\kappa,2\nu-\kappa,\beta,T}<\infty$.
  
  Let now $(u_0^n)_{n\geq1}$ be a sequence of smooth random
  fields (obtained for instance from $u_0$ by convolution)
  such that $u_0^n\to u_0$, in the sense that
  $\sup_j j^{-1/2}2^{-j}\|\Delta_j(u_0^n-u_0)\|_\infty\to0$,
  with similar convergence for $\etazero_n$ and
  $\etaunodue_n$ (in the appropriate norms), where
  $\etazero$ and $\etaunodue$ are the stochastic objects
  derived from $u_0^n$. If, for every $n$, $v^n$ is the
  solution of \eqref{e:burgers2} (with $\etazero_n$ and
  $\etaunodue_n$), then there is a {a.\,s.} positive random
  time $T$ such that $\sup_n \|v^n\|_{\kappa,2\nu-\kappa,\beta,T}<\infty$.
  Set $w_n=v_n-v$, $p_n=\etazero_n-\etazero$,
  $q_n=\etaunodue_n-\etaunodue$, then
  \[
    R_n(v_n)-R(v)
      = \Vs(v_n+v,w_n) + q_n + 2\Vs(w_n\pgeq\etazero_n)+2\Vs(v\pgeq p_n),
  \]
  where $R_n$ is the remainder with $\etazero_n$ and $\etaunodue_n$.
  and, using estimate similar to those in the proof of
  Lemma~\ref{l:remainder}, we see that
  $\|R_n(v_n)-R(v)\|_{\kappa,\nu,\beta,T}\to0$.
  Likewise, since
  \[
    \begin{multlined}[.9\linewidth]
      w_n
        = 4\Vs(\Vs(w_n\plt\etazero_n)\plt\etazero_n)
          + 4\Vs(\Vs(v\plt p_n)\plt\etazero_n) +\\
          + 4\Vs(\Vs(v\plt \etazero_n)\plt p_n)
          + 2\Vs((R_n(v_n)-R(v))\plt\etazero_n) +\\
          + 2\Vs(R(v)\plt p_n)
          + R_n(v_n)-R(v),
    \end{multlined}
  \]
  we have that $\|w_n\|_{\kappa,2\nu-\kappa,\beta,T}\to0$
  (using also estimates as those in Theorem~\ref{t:local}).
  
  Now
  \[
    v_n
      = \Vs(v_n,v_n) + 2\Vs(v_n,\etazero_n) + \etaunodue_n
      = R_n(v_n) + 2\Vs(v_n\plt\etazero_n)
  \]
  and it remains to show that the term $\Vs(v_n\plt\etazero_n)$
  converges. Indeed,
  \[
    \begin{multlined}[.98\linewidth]
    t^\beta\|\Delta_j\Vs(w_n\plt\etazero)\|_\infty\approx\\
      \approx \|w_n\|_{\kappa,2\nu-\kappa,\beta,T}t^\beta\sqrt j 2^{2j}
        \int_0^t \e^{-2^{2j}(t-s)}\e^{-2^{2j}s}\sum_{n=0}^{j-2}
        s^{-\beta}n^{-\kappa}\ell(s)^{-\nu}\,ds\lesssim\\
      \lesssim t H_{\frac12,2,2}(t)\ell(t)^{-\nu}
        \|w_n\|_{\kappa,2\nu-\kappa,\beta,T}
      \lesssim\|w_n\|_{\kappa,2\nu-\kappa,\beta,T},
    \end{multlined}
  \]
  and likewise for $\Vs(v_n,p_n)$.
\end{proof}
\begin{remark}
  We remark that an attempt to run a fixed point in the space
  where the norm $\|\cdot\|_{\kappa,\nu,\beta,T}$ would fail
  when trying to prove the self mapping property for the term
  $\Vs(v\plt\etazero)$.
\end{remark}
We conclude with an elementary analytical lemma.
\begin{lemma}\label{l:technical}
  Set for every $\nu\in\R$, $p\geq0$, $\tau>0$, and $t>0$,
  \[
    G_{\nu,p,\tau}(t)
      = \sum_{n=1}^\infty n^\nu 2^{pn}\e^{-2^{\tau n}t},
        \qquad
    H_{\nu,p,\tau}(t)
      = \sup_{n\geq1} n^\nu 2^{pn}\e^{-2^{\tau n}t}.
  \]
  Then for $p>0$ and $\nu\in\R$,
  \[
    H_{\nu,p,\tau}
      \leq G_{\nu,p,\tau}(t)
      \lesssim t^{-\frac{p}\tau}\ell(t)^\nu.
  \]
  Moreover, if $p=0$, $\nu\in\R$,
  \[
    H_{\nu,0,\tau}
      \leq \ell(t)^{\nu_+}.
  \]
\end{lemma}

\end{document}